\documentclass[manuscript]{acmart}
\usepackage{dsfont}

\newcommand{\ee}{{\mathrm e}}
\newcommand{\sujsqd}{\text{SUJSQ}^{\textup{det}}(\d)}
\newcommand{\sujsqe}{\text{SUJSQ}^{\textup{exp}}(\d)}
\newcommand{\aujsqd}{\text{AUJSQ}^{\textup{det}}(\d)}
\newcommand{\aujsqe}{\text{AUJSQ}^{\textup{exp}}(\d)}
\newcommand{\sujsqdi}{\text{SUJSQ}^{\textup{det,idle}}(\d)}
\newcommand{\sujsq}{\text{SUJSQ}(\d)}
\newcommand{\aujsq}{\text{AUJSQ}(\d)}
\renewcommand{\d}{\delta}

\newcommand{\e}{\text{e}}

\newcommand{\indi}[1]{{\mathds 1}{\left\{{\displaystyle#1}\right\}}}
\newcommand{\indis}[1]{{\mathds 1}{\{{\displaystyle#1}\}}}

\newtheorem{algorithm}{Algorithm}

\graphicspath{{Figures/},{./}}

\def\BibTeX{{\rm B\kern-.05em{\sc i\kern-.025em b}\kern-.08emT\kern-.1667em\lower.7ex\hbox{E}\kern-.125emX}}
    
\setcopyright{acmlicensed}
\acmJournal{POMACS}
\acmYear{2019} \acmVolume{3} \acmNumber{1} \acmArticle{4} \acmMonth{3} \acmPrice{15.00}\acmDOI{10.1145/3311075}

\begin{document}

%
\title{Hyper-Scalable JSQ with Sparse Feedback}

%
\author{Mark van der Boor}
\affiliation{%
  \institution{Eindhoven University of Technology}
  \streetaddress{P.O. Box 513}
  \city{Eindhoven}
  \postcode{5600 MB}
  \country{The Netherlands}
}
\author{Sem Borst}
\affiliation{%
  \institution{Eindhoven University of Technology}
 \country{The Netherlands}
}
\affiliation{
  \institution{Nokia Bell Labs}
  \streetaddress{P.O. Box 636}
  \city{Murray Hill}
  \state{NJ}
  \postcode{07974}
  \country{USA}
}

\author{Johan van Leeuwaarden}
\affiliation{%
  \institution{Eindhoven University of Technology}
 \country{The Netherlands}
}

%
\renewcommand{\shortauthors}{Van der Boor, Borst, Van Leeuwaarden}

%
\begin{abstract}
		Load balancing algorithms play a vital role in enhancing performance in data centers and cloud networks.
		Due to the massive size of these systems, scalability challenges,
		and especially the communication overhead associated with load
		balancing mechanisms, have emerged as major concerns.
		Motivated by these issues, we introduce and analyze a novel class
		of load balancing schemes where the various servers provide occasional
		queue updates to guide the load assignment.
		
		We show that the proposed schemes strongly outperform JSQ($d$)
		strategies with comparable communication overhead per job, and can achieve
		a vanishing waiting time in the many-server limit with just one
		message per job, just like the popular JIQ scheme. 
		The proposed schemes are particularly geared however towards the
		sparse feedback regime with less than one message per job,
		where they outperform corresponding sparsified JIQ versions. 
		
		We investigate fluid limits for synchronous updates as well as
		asynchronous exponential update intervals.
		The fixed point of the fluid limit is identified in the latter case,
		and used to derive the queue length distribution.
		We also demonstrate that in the ultra-low feedback
		regime the mean stationary waiting time tends to a constant in the
		synchronous case, but grows without bound in the asynchronous case.
\end{abstract}

%
%


%


%

%
\maketitle

\section{Introduction}

{\em Background and motivation.}
We introduce and analyze hyper-scalable load
balancing algorithms that only involve minimal communication
overhead and yet deliver excellent performance.
Load balancing algorithms play a key role in efficiently distributing
jobs (e.g.~compute tasks, database look-ups, file transfers) among
servers in cloud networks and data centers \cite{duet14,MSY12,ananta13}.
Well-designed load balancing schemes provide an effective mechanism
for improving performance metrics in terms of response times
while achieving high resource utilization levels.
Besides these typical performance criteria, communication overhead
and implementation complexity have emerged as equally crucial
attributes, due to the immense size of cloud networks and data centers.
These scalability challenges have fueled a strong interest in the
design of load balancing algorithms that provide robust performance
while only requiring low overhead.

We focus on a basic scenario of $N$~parallel
identical servers, exponentially distributed service requirements,
and a service discipline at each server that is oblivious to the
actual service requirements (e.g.~FCFS).
In this canonical case, the Join-the-Shortest-Queue (JSQ) policy
has strong stochastic optimality properties, and in particular
minimizes the overall mean delay among the class of non-anticipating
load balancing policies that do not have any advance knowledge
of the service requirements \cite{EVW80,Winston77}.

In order to implement the JSQ policy, a dispatcher requires
instantaneous knowledge of the queue lengths at all the servers,
which may give rise to a substantial communication burden,
and not be scalable in scenarios with large numbers of servers.
The latter issue has motivated consideration of so-called JSQ($d$)
strategies, where the dispatcher assigns incoming jobs to a server
with the shortest queue among $d$~servers selected uniformly at random.
This involves an exchange of $2 d$ messages per job (assuming
$d \geq 2$), and thus greatly reduces the communication overhead
compared to the JSQ policy when the number of servers~$N$ is large.
At the same time, results in Mitzenmacher~\cite{Mitzenmacher01}
and Vvedenskaya {\em et al.}~\cite{VDK96} indicate that even
a value as small as $d = 2$ yields significant performance
improvements as $N \to \infty$ compared to a purely random
assignment scheme ($d = 1$).
This is commonly referred to as the ``power-of-two'' effect.

Although JSQ($d$) strategies provide notably better waiting-time
performance than purely random assignment, they lack the ability
of the conventional JSQ policy to drive the waiting time to zero
in the many-server limit.
Moreover, while JSQ($d$) strategies notably reduce the amount
of communication overhead compared to the full JSQ policy,
the two-way delay incurred in obtaining queue length information
still directly adds to the waiting time of each job.
The latter achilles heel of `push-based' strategies is eliminated
in `pull-based' strategies where servers pro-actively provide queue
length information to the dispatcher.
A particularly popular pull-based strategy is the so-called
Join-the-Idle-Queue (JIQ) scheme \cite{BB08,LXKGLG11}. Servers advertise their availability to the dispatcher whenever they become idle, which involves no more than one message per job to send a job to an available idle server.
This pull-based strategy has the ability of the full JSQ policy to achieve a zero
waiting time in the many-server limit~\cite{Stolyar15}. A pull-based implementation for the JSQ policy exists but it leads to more frequent communication requirements or larger communication messages.

The superiority of the JIQ scheme over JSQ($d$) strategies in terms
of performance and communication overhead is owed to the state
information stored at the dispatcher.
Results in~\cite{GTZ16} imply that a vanishing waiting time can
only be achieved with finite communication overhead per job
when allowing memory usage at the dispatcher, and further suggest
that one message per job is a minimal requirement in a certain sense.
However, even just one message per job may still be prohibitive,
especially when jobs do not involve big computational tasks,
but small data packets which require little processing,
e.g.~in IoT cloud environments.
In such situations the sheer message exchange in providing queue
length information may be disproportionate to the actual amount
of processing required. 




{\em Hyper-scalable algorithms.} 
Motivated by the above issues, we propose and examine a novel class
of load balancing schemes which also leverage memory at the dispatcher,
but allow the communication overhead to be seamlessly adapted
and reduced below that of the JIQ scheme. The basic scheme is as follows:

\begin{algorithm}[Basic hyper-scalable scheme]\label{hsalg}
The dispatcher forwards incoming jobs to the server with the lowest
queue estimate.  The dispatcher maintains an estimate for every server
and increments these estimates for every job that is assigned.
 Status updates of servers occur at rate $\delta$ per server,
and update the estimate that the dispatcher has at its disposal to the
actual queue length.
\end{algorithm}

Several aspects of Algorithm \ref{hsalg} are flexible (regarding the implementation and the status updates) and four different schemes that obey the rules of Algorithm \ref{hsalg} will be introduced. There are many more schemes that could be of interest, for example a scheme where the queue estimate is not an upper-bound but mimics the expected value of the queue length. While natural, these schemes are beyond the scope of the current paper.

When the update frequency per server is $\delta$ and $\lambda<1$ denotes the arrival rate per server, the number of messages per job is $\delta / \lambda$, which can be easily tuned by varying the value of $\delta$. Since all queue lengths are updated (on average) once every $1/\d$ time units, this gives $\d N$ queue-updates per time unit. Note that this algorithm can be modeled as a strictly push-based scheme (where the dispatcher requests the queue lengths of the servers), as well as a strictly pull-based scheme (where each server sends its queue length to the dispatcher, using an internal clock).

The JSQ($d$) scheme, when implemented in a push-based manner, requires $2d$ message exchanges per job, which amounts to $2 \lambda d N$ messages per time unit and is not scalable. However, when servers actively update their queue lengths to the dispatcher in the JSQ scheme or their idleness in the pull-based JIQ scheme, one needs less communication. In this case, any departing job needs to trigger the server to send an update to the dispatcher. This pull-based implementation requires one message of communication per job or $\lambda N$ per time unit. When queue lengths are large, not even all departing jobs need to trigger the server to send an update for the JIQ policy, which reduces the communication per job slightly. We thus conclude that the tunable communication overhead of $\d/\lambda$ per job (doubled when implemented in a push-based manner) of Algorithm \ref{hsalg} is comparable with pull-based JSQ, JIQ and JSQ($d$). Moreover,  Algorithm \ref{hsalg}  becomes more scalable for small values of $\d$, especially in the $\d<\lambda$-regime. Here an important reference point is that the pull-based JIQ scheme has at most one update per job. By hyper-scalable schemes we mean schemes that can be implemented with  $\d\leq 1$, and preferably with $\d \ll 1$. 

We introduce four hyper-scalable schemes that obey the rules of Algorithm \ref{hsalg} but differ in when the status-updates are sent. When the updates sent by the servers to the dispatcher are synchronized, we denote the scheme by $\sujsq$, Synchronized-Updates Join-the-Shortest-Queue. Similarly we introduce $\aujsq$ (Asynchronized-Updates Join-the-Shortest-Queue), which is used when the updates are asynchronous. We then add an exp-tag whenever the time between two updates is exponentially distributed (with parameter $\d$ and mean $1/\d$) and a det-tag when the time between the updates is constant ($1/\d$). This gives rise to four schemes; $\sujsqd$ , $\sujsqe$, $\aujsqd$ and $\aujsqe$. 

We show that the four schemes can achieve a vanishing waiting time in the many-server limit with just one message per job, just like JIQ. The proposed schemes are particularly geared however towards the sparse feedback regime with less than message per job, 
where they outperform corresponding sparsified JIQ versions. With fluid limits we demonstrate that in the ultra-low feedback regime the mean stationary waiting time tends to a constant in the synchronous case, but grows without bound in the asynchronous case. A more detailed overview of our key finding is presented in the next section. 



{\em Discussion of additional related work.}
As mentioned above, Mitzenmacher~\cite{Mitzenmacher01}
and Vvedenskaya {\em et al.}~\cite{VDK96} established mean-field
limit results for JSQ($d$) strategies.
These results indicate that for any subcritical arrival rate
$\lambda < 1$, the tail of the queue length distribution at each
individual server exhibits super-exponential decay,
and thus falls off far more rapidly than the geometric decay
in case of purely random assignment.
Similar power-of-$d$ effects have been demonstrated for
heterogeneous servers, non-exponential service requirements
and loss systems \cite{BLP10,BLP12,MKM15,MKMG15,MM14a,XDLS15}.
For no single value of~$d$, however, a JSQ($d$) strategy can rival
the JIQ scheme, which simultaneously achieves low communication
overhead and asymptotically optimal performance by leveraging
memory at the dispatcher~\cite{GTZ16,Stolyar15}.
The only exception arises for batches of jobs when the value of~$d$
and the batch size grow suitably large, as can be deduced from results
in~\cite{YSK15}, but we do not leverage batches in the current paper.
As we will show, the hyper-scalable schemes proposed in the present
paper are like the JIQ scheme superior to JSQ$(d$) strategies,
and also beat corresponding sparsified JIQ versions in the regime
$\d < 1$, which is particularly relevant from a scalability standpoint. 

Many popular schemes have also been analyzed in the Halfin-Whitt heavy-traffic regime \cite{MBLW16} and in the non-degenerate slowdown regime \cite{GW17}, in which the JIQ scheme is not necessarily optimal.

The use of memory in load balancing has been studied
in \cite{AGL10,Mitzenmacher02} but mostly in a `balls-and-bins' context
as opposed to the queueing scenario that we consider.
The work in~\cite{Mitzenmacher00} considers a similar setup as ours,
and examines how much load balancing degrades when old information is used.

In contrast, our focus is on improving the performance by using
non-recent information, similarly to \cite{AD18}.
A general framework for deriving fluid limits in the presence of memory
is described in~\cite{LN13}, but assumes that the length of only one queue
is kept in memory, while we allow for all queue lengths to be tracked. 

{\em Organization of the paper.}
In Section~\ref{sec:results} we discuss our key findings and contributions for the hyper-scalable schemes, obtained through fluid-limit analysis and extensive simulations. In Section~\ref{notaprel} we introduce some useful notation and preliminaries, before we turn to a comprehensive analysis of the synchronous and asynchronous cases through the lens of fluid limits in Sections~\ref{sec:sync} and~\ref{sec:async}, respectively.  In Section~\ref{sec:conc} we conclude with some summarizing remarks and topics for further research.


\section{Key findings and contributions}
\label{sec:results}

The precise model we consider consists of $N$ parallel identical servers and one dispatcher. Jobs arrive at the dispatcher as a Poisson process of rate $\lambda N$, where $\lambda$ denotes the job arrival rate per server. Every job is dispatched to one of the servers, after which it joins the queue of the server if the server is busy, or will start its service when the server is idle. The job processing requirements are independent and exponentially distributed with unit mean at each of the servers. We consider several load balancing algorithms for the dispatching of jobs to servers, including the hyper-scalable schemes $\sujsqd$, $\sujsqe$, $\aujsqd$ and $\aujsqe$ described in the introduction. In the simulation experiments we will also briefly consider $\sujsqdi$, which is similar to $\sujsqd$, except that only idle servers send notifications. We now present the results from simulation studies and the fluid-limit and fixed-point analysis in Sections~\ref{sec:sync} and~\ref{sec:async}.

\subsection{Large-system performance}

In order to explore the performance of the hyper-scalable algorithms in the many-server limit $N\to\infty$, we investigate fluid limits. We analyze their behavior and fixed points, and use these to derive results for the system in stationarity as function of the update frequency~$\d$.



{\em Asymptotically optimal feedback regime.}
Using fluid-limit analysis, we prove that the proposed schemes can
achieve a vanishing waiting time in the many-server limit when the
update frequency $\delta$ exceeds $\lambda / (1 - \lambda)$.
In case servers only report zero queue lengths and suppress 
updates for non-zero queues, the update frequency required for
a vanishing waiting time can in fact be lowered to just~$\lambda$,
matching the one message per job involved in the JIQ scheme. 


\begin{figure*}[t]

\begin{center}
\includegraphics[width=\linewidth]{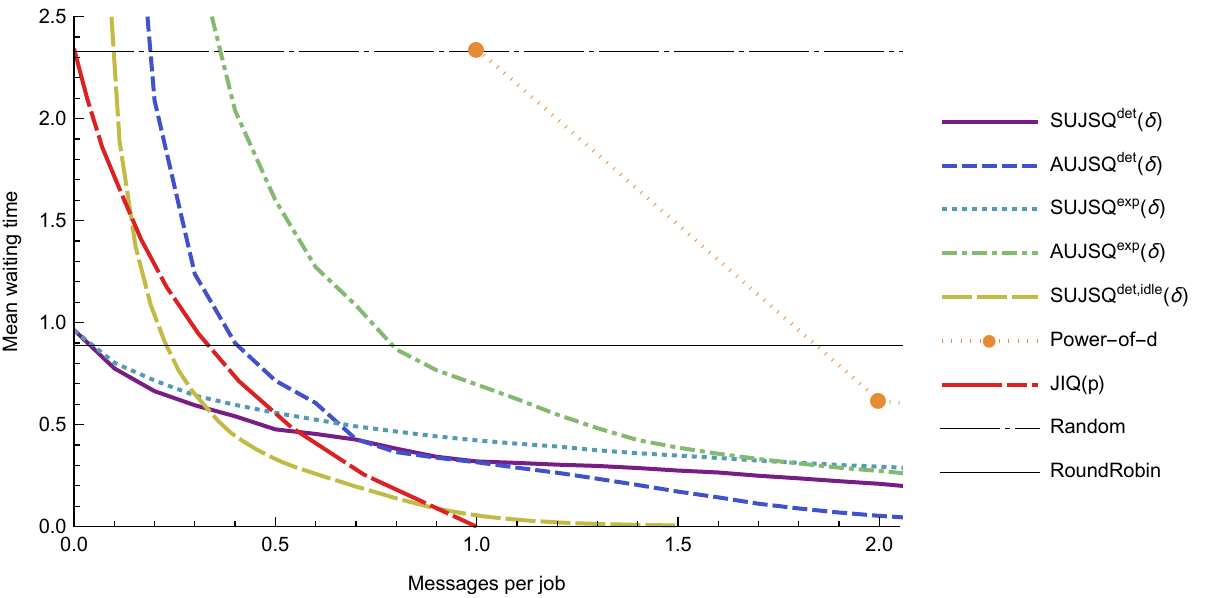}
\end{center}

\caption{Mean waiting times for different schemes (with various parameters) for $\lambda=0.7$ and $N=200$ obtained via a discrete-event simulation. For each of the schemes, the parameter~$\delta$, $d$ or~$p$ is varied and the communication in terms of messages per job and the waiting times of jobs are tracked in the simulation.}
\label{fig:ds}
\end{figure*}

\label{benchmarks}

{\em Sparse feedback regime.}
Figure~\ref{fig:ds} displays results from extensive simulations and shows the mean
waiting time as function of the number of messages per job. This number is proportional to the update frequency~$\d$,
and equals $\d/\lambda$ for the  four hyper-scalable schemes. We also show results for  JIQ$(p)$, a sparsified version of JIQ, where a token is sent
to the dispatcher with probability~$p$ whenever a server becomes idle.
Random refers to the scheme where every job is assigned to
a server selected uniformly at random, and Round-Robin assigns the
$i$-th arriving job to server $1 + i \mod N$.

For the sparse feedback regime when $\d < 0.5$ we see that the schemes $\sujsqd$ and $\sujsqe$ outperform JIQ$(p)$. 
Also observe that $\sujsqdi$, the scheme in which only idle servers send reports, achieves a near-zero waiting time with just one message per job, just like the JIQ scheme, and outperforms JIQ($p$) across most of the relevant domain $\d < 0.5$.
However, as $\delta \downarrow 0$ the waiting time grows without bound, since estimates will grow large due to lack of updates, which causes servers that are reported idle in the latest update to receive many jobs in succession.

{\em Ultra-low feedback regime.}
We examine the performance in the ultra-low feedback regime
where the update frequency~$\delta$ goes to zero,
and in particular establish a somewhat counter-intuitive dichotomy.
When all status-updates occur synchronously, the behavior of each
of the individual queues approaches that of a single-server queue with
a near-deterministic arrival process and exponential service times,
with the mean stationary waiting time tending to a finite constant.
In contrast, for asynchronous updates, the individual queues
experience saw-tooth behavior with oscillations and waiting times
that grow without bound.

\subsection{Synchronize or not?}
\label{subsec:ressync}

In case of synchronized updates, the dispatcher will update the queue lengths of the servers simultaneously. Thus, just after an update moment, the dispatcher has a perfect view of the status of all servers and it will dispatch jobs optimally. After a while, the estimates will start to deviate from the actual queue lengths, so that the scheme no longer makes (close to) optimal decisions. With asynchronous updates servers send updates at independent times, which means that some of the estimates may be very accurate, while others may differ significantly from the actual queue lengths.

{\em Round-Robin resemblance.}
We find that both $\sujsqd$ and $\sujsqe$ resemble Round-Robin as the
update frequency~$\delta$ approaches zero, and are the clear
winners in the ultra-low feedback regime, which is crucial
from a scalability perspective (see Figure~\ref{fig:ds}). To understand the resemblance with Round-Robin, notice that the
initial queue lengths after an update will be small compared
to the number of arrivals until the next update.
Thus soon after the update the dispatcher will essentially start
forwarding jobs in a (probabilistic) Round-Robin manner.
Specifically, most servers will have equal queue estimates at certain points in time, and they will each receive one job every $\lambda$ time units, but in a random order. This pattern repeats itself and resembles Round-Robin, where the difference of received jobs among servers can be at most one.

{\em Dichotomy.}
In  Figure~\ref{fig:ds} we also see that while $\aujsqd$ outperforms the synchronous variants for
large values of the update frequency~$\delta$, it produces a mean
waiting time that grows without bound as $\delta$ approaches zero.
The latter issue also occurs for $\aujsqe$ and render the
asynchronous versions far inferior in the ultra-low feedback regime
compared to both synchronous variants. To understand this remarkable dichotomy, notice that
 queue estimates must inevitably grow to increasingly large values
of the order $\lambda / \delta$ and significantly diverge from the
true queue lengths as the update frequency becomes small,
both in the synchronous and asynchronous versions.
However, in the synchronous variants the queue estimates will all
be lowered and updated to the true queue lengths simultaneously,
prompting the dispatcher to evenly distribute incoming jobs over time.
In contrast, in the asynchronous versions, a server will be the
only one with a low queue estimate right after an update,
and almost immediately be assigned a huge pile of jobs to bring
its queue estimate at par, resulting in oscillatory effects.
This somewhat counter-intuitive dichotomy reveals that the
synchronous variants behave benignly in the presence of outdated
information, while the asynchronous versions are adversely impacted.


\section{Notation and preliminaries}
\label{notaprel}
In this section we introduce some useful notation and preliminaries
in preparation for the fluid-limit analysis in Sections~\ref{sec:sync}
and~\ref{sec:async}.
Recall that all the servers are identical and the dispatcher only
distinguishes among servers based on their queue estimates and does
not take their identities into account when forwarding jobs.
Hence we do not need to keep track of the state of each individual server,
but only count the number of servers that reside in a particular state.
Specifically, we will denote by $Y_{i,j}(t)$ the number of servers
with queue length~$i$ (including a possible job being served)
and queue estimate $j \geq i$ at the dispatcher at time~$t$.
Further denote by $V_i = \sum_{l = i}^{\infty} Y_{il}$
and $W_j = \sum_{k = 0}^{j} Y_{kj}$ the total number of servers
with queue length~$i$ and queue estimate~$j$, respectively,
when the system is in state~$Y$.

In order to analyze fluid limits in a regime where the number
of servers~$N$ grows large, we will consider a sequence of systems indexed
by~$N$, and attach a superscript~$N$ to the associated state variables.
We specifically introduce the fluid-scaled state variables
$y_{i,j}^N(t) = Y_{i,j}^N(t) / N$, representing the fraction of servers
in the $N$-th system with true queue length~$i$ and queue estimate
$j \geq i$ at the dispatcher at time~$t$, and assume that the sequence
of initial states is such that $y^N(0) \to y^\infty$.
Any (weak) limit $\{y(t)\}_{t \geq 0}$ of the sequence
$(\{y^N(t)\}_{t \geq 0})_{N \geq 1}$ as $N \to \infty$ will
be called a fluid limit.
Fluid limits do not only yield tractability, but also provide
a relevant tool to investigate communication overhead
and scalability issues which are inherently tied to scenarios
with massive numbers of servers.

Let $m(Y) = \min\{j: W_j > 0\}$ be the minimum queue estimate
among all servers when the system is in state~$Y$.
When a job arrives and the system is in state~$Y$, it is dispatched to one
of the $W_{m(Y)}$ servers with queue estimate $m(Y)$ selected
uniformly at random, so it joins a server with queue length $i \leq m(Y)$
with probability $Y_{i,m(Y)} / W_{m(Y)}$.
Because of the Poisson arrival process, transitions from a state~$Y$
to a state~$Y'$ with $Y_{i,j}' = Y_{i,j} - 1$
and $Y_{i+1,j+1}' = Y_{i+1,j+1} + 1$ thus occur at rate
$\lambda N p_{i,j}(Y)$,
with $p_{i,j}(Y) = \indi{j = m(Y)} Y_{i,j} / W_j$, $i = 0, \dots, j$.
Due to the unit-exponential processing requirements, transitions
from a state~$Y$ to a state~$Y'$ with $Y_{i,j}' = Y_{i,j} - 1$
and $Y_{i-1,j}' = Y_{i-1,j} + 1$ occur at rate $Y_{i,j}$, $i = 1, \dots, j$. For notational compactness, we further omit the dependence of $Y$ for $m(Y)$, $p_{i,j}(Y)$ and $q_{i,j}(Y)$ and instead write $m(t)$, $p_{i,j}(t)$ and $q_{i,j}(t)$ as they only depend on $t$ though $y(t)$.

In order to specify the transitions due to the updates, we need to
distinguish between the synchronous and the asynchronous case.

\subsubsection*{Synchronous updates}

Whenever a synchronous update occurs and the system is in state~$Y$,
a transition occurs to state~$Y'$ with $Y_{ii}' = V_i$ and $Y_{ij}' = 0$
for $i = 0, \dots, j - 1$.
Note that these transitions only occur in a Markovian fashion when the
update intervals are exponentially distributed.
When the update intervals are non-exponentially distributed,
$\{Y(t)\}_{t \geq 0}$ is not a Markov process, but the evolution
between successive updates is still Markovian.

\subsubsection*{Asynchronous updates}

When the system is in state~$Y$ and a server with queue length~$i$
and queue estimate $j > i$ sends an update to the dispatcher,
a transition occurs to a state~$Y'$ with $Y'_{i,j} = Y_{i,j} - 1$
and $Y_{i,i}' = Y_{i,i} + 1$.
Note that these transitions only occur in a Markovian fashion when the
update intervals are exponentially distributed.
When the update intervals are non-exponentially distributed,
$\{Y(t)\}_{t \geq 0}$ is not a Markov process, and in order to obtain
a Markovian state description, the state variables $Y_{ij}$ would
in fact need to be augmented with continuous variables keeping track
of the most recent update moments for the various servers.

\section{Synchronous updates}
\label{sec:sync}

In this section we examine the fluid limit for synchronous updates.
In Subsection~\ref{subsec:dynasync} we provide a description of the
fluid-limit trajectory, along with an intuitive interpretation,
numerical illustration and comparison with simulation.
In Subsection~\ref{subsec:fpsync} the fluid-limit analysis will be
used to derive a finite upper bound for the queue length on fluid scale
for any given update frequency $\d>0$ (Proposition~\ref{prop:sync2})
and to show that in the long term queueing vanishes on fluid level for
a sufficiently high update frequency~$\d$ (Proposition~\ref{prop:sync1}).

\subsection{Fluid-limit dynamics}
\label{subsec:dynasync}

The fluid limit $y(t)$ (in between successive update moments)
satisfies the system of differential equations
\begin{equation}
\begin{split}
\label{eq:fl}
\frac{dy_{i,j}(t)}{dt} =&
~y_{i+1,j}(t) \mathds{1}\{i < j\} - y_{i,j}(t) \mathds{1}\{i>0\} \\
&+
\lambda p_{i-1,j-1}(t) \mathds{1}\{i > 0\} - \lambda p_{i,j}(t),
\end{split}
\end{equation}
where
\[
p_{i,j}(t) = \frac{y_{i,j}(t)}{w_j(t)} \mathds{1}\{m(t)=j\}
\]
denotes the fraction of jobs assigned to a server with queue
length~$i$ and queue estimate~$j$ in fluid state~$y$ at time $t$,
with $w_j(t) = \sum_{k = 0}^{j} y_{k,j}(t)$ denoting the fraction
of servers with queue estimate~$j$ and $m(t) = \min(j | w_j(t) > 0)$
the minimum queue estimate in fluid state~$y$ at time $t$.
At an update moment~$T$, the fluid limit shows discontinuous behavior,
with $y_{i,i}(T) = v_i(T^-)$ and $y_{i,j}(T) = 0$ for all $i < j$,
with $v_i(t) = \sum_{l = i}^{\infty} y_{i,l}(t)$ denoting the fraction
of servers with queue length~$i$ in fluid state~$y$ at time $t$.

An informal outline of the derivation of the fluid limit as stated
in~\eqref{eq:fl} is provided in Appendix~\ref{sketchsync}.

\subsubsection{Interpretation}

The above system of differential equations may be heuristically
explained as follows.
The first two terms correspond to service completions at servers
with $i + 1$ and $i$~jobs, which result in an increase and decrease
in the fraction of servers with $i$~jobs, respectively.
The third and fourth terms reflect the job assignments to servers
with the minimum queue estimate $m(t)$.
The third term captures the resulting increase in the fraction
of servers with queue estimate $m(t) + 1$, while the fourth term
captures the corresponding decrease in the fraction of servers
with queue estimate $m(t)$.

Summing the equations \eqref{eq:fl} over $i = 0, 1, \dots, j$ yields
\begin{align}
\begin{split}\label{eq:syncdw}
\frac{dw_j(t)}{dt} &=
\lambda [\mathds{1}\{m(t)=j-1\} - \mathds{1}\{m(t)=j\}]\\
&=
\left\{\begin{array}{ll} - \lambda & j = m(t) \\
\lambda & j = m(t) + 1 \\ 0 & j \neq m(t), m(t) + 1\end{array}\right.
\end{split}
\end{align}

\noindent
reflecting that servers with the minimum queue estimate $m(t)$
are assigned jobs, and flipped into servers with queue estimate
$m(t) + 1$, at rate~$\lambda$, and that $m(t)$ can only
increase between successive update moments.
Also note that the derivative of $y_{i,j}$ is continuous in~$t$,
except at those times~$t$ where $m(t)$ increases,
and that $y_{i,j}(t)$ is continuous in between updates since
$dy_{i,j}(t)/dt$ is bounded.

\begin{figure}
    \centering
    \begin{minipage}{0.49\textwidth}
        \centering
        \includegraphics[width=\textwidth]{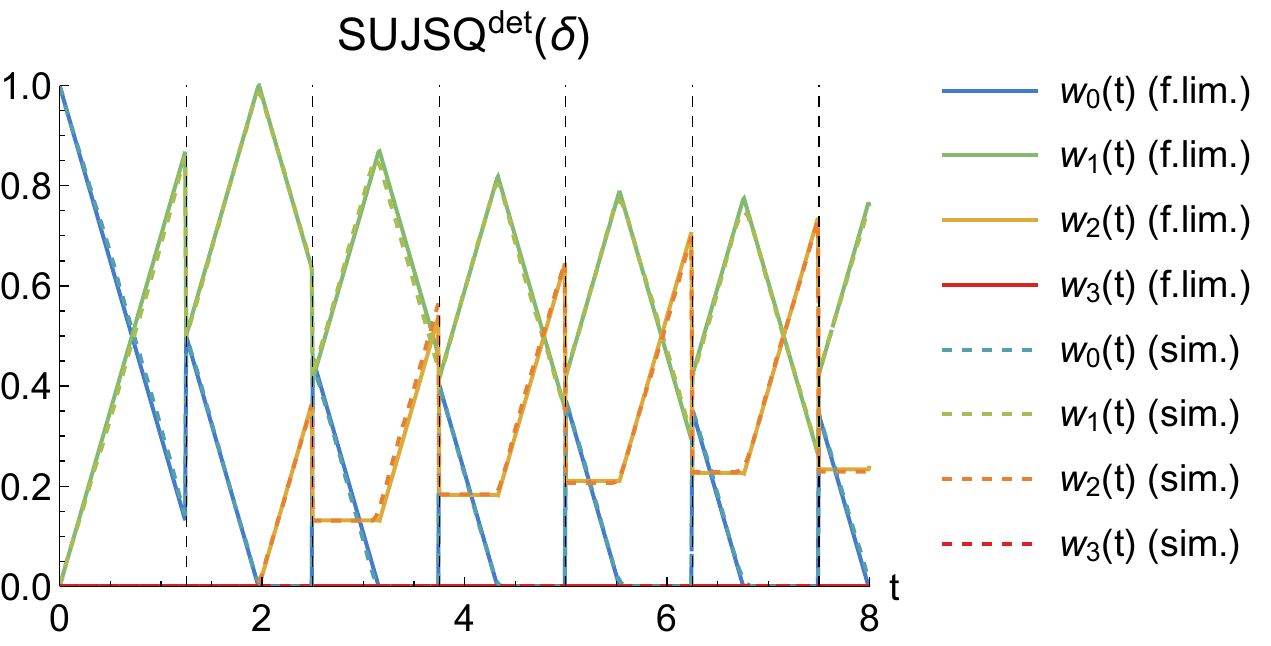} 
        \caption{Numerical emulation of the fluid limit for SUJSQ$^{\textup{det}}(0.85)$
        		and $\lambda = 0.7$, accompanied by simulation results for $N = 1000$,
        		averaged over 10~runs.}
        	\label{fig:ss11}
    \end{minipage}\hfill
    \begin{minipage}{0.49\textwidth}
        \centering
        \includegraphics[width=\textwidth]{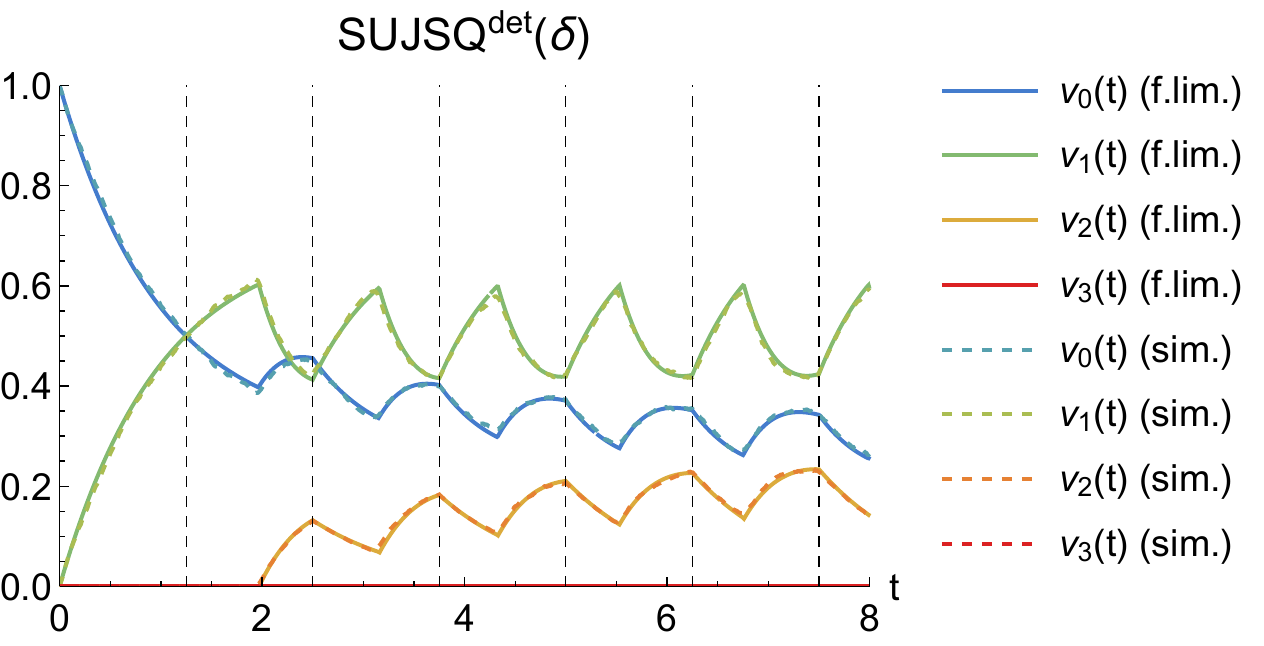} 
        	\caption{Numerical emulation of the fluid limit for SUJSQ$^{\textup{det}}(0.85)$
        		and $\lambda = 0.7$, accompanied by simulation results for $N = 1000$,
        		averaged over 10~runs.}
        	\label{fig:ss12}
    \end{minipage}
\end{figure}

\subsubsection{Numerical illustration and comparison with simulation}
\label{simusync}

Figures~\ref{fig:ss11}--\ref{fig:ss14} show the fluid-limit
trajectories $y(t)$ as governed by the differential equations
in~\eqref{eq:fl} for $\sujsqd$, along with (fluid-scaled) variables
$y_{i,j}^N(t)$ obtained through stochastic simulation for a system
with $N = 1000$ servers and averaged over 10~runs.
Observe that the simulation results are nearly indistinguishable
from the fluid-limit trajectories,
which is in line with broader findings concerning the accuracy
of fluid and mean-field limits~\cite{Gast17,Ying16}.

Update moments at times $1/\delta, 2/\delta, \hdots$ are marked
by vertical dotted lines.
These time points can also be easily recognized by the jumps in the
fraction of servers $w_j(t)$ that have queue estimate~$j$.
Moreover, the fraction of servers $v_i(t)$ with queue length~$i$
is not differentiable in these points as well as other moments when
the minimum queue estimate changes.

Qualitatively similar results are observed for $\sujsqe$,
where the updates occur at irregular moments.
The paths still follow similar saw-tooth patterns, and the dynamics
between updates are identical, as reflected in the differential
equations in~\eqref{eq:fl}.
In particular, the fraction of servers with minimum queue estimate
decreases linearly between updates, and the estimate drastically
changes at update moments.
The results are displayed in Figures~\ref{fig:ss21}--\ref{fig:ss24},
which are deferred to Appendix~\ref{app:simressync} because of space
constraints.

In Figures~\ref{fig:ss11} and~\ref{fig:ss12},
$\delta = 0.85 < \lambda/(1-\lambda) = 7/3$,
while $\delta = 2.5 > 7/3$ in Figures~\ref{fig:ss13} and~\ref{fig:ss14}.
Since in the first scenario there are moments at which $w_0(t)$ is zero,
some jobs are sent to servers with one job, so that servers sometimes have two jobs, which means that queueing does not vanish as $N \to \infty$.
In contrast, in the second scenario, $w_0(t)$ is strictly positive
at all times and no servers appear with two or more jobs,
which implies that no queueing occurs at fluid level as $N \to \infty$.
We will return to this dichotomy in Proposition~\ref{prop:sync1}.

\subsection{Performance in the fluid limit}
\label{subsec:fpsync}

We will now use the fluid limit~\eqref{eq:fl} to gain some insight
in the performance of the $\sujsqe$ scheme.
\eqref{eq:syncdw} shows that no fixed point can exist as $w_j(t)$ has a non-zero constant derivative. 
We will establish however in Proposition~\ref{prop:sync2} that
for any positive update frequency the queue lengths on fluid scale are
essentially bounded by a finite constant.
Furthermore, in Proposition~\ref{prop:sync1} we demonstrate
that when the update frequency is above a specific threshold,
queueing basically vanishes on fluid level in the long term.

Consider the average queue length on fluid scale, denoted and defined by
\[
Q(t) = \sum_{i = 0}^{\infty} i v_i(t) =
\sum_{i = 0}^{\infty} i \sum_{j = i}^{\infty} y_{i,j}(t).
\]

\begin{figure}
    \centering
    \begin{minipage}{0.49\textwidth}
        \centering
        \includegraphics[width=\textwidth]{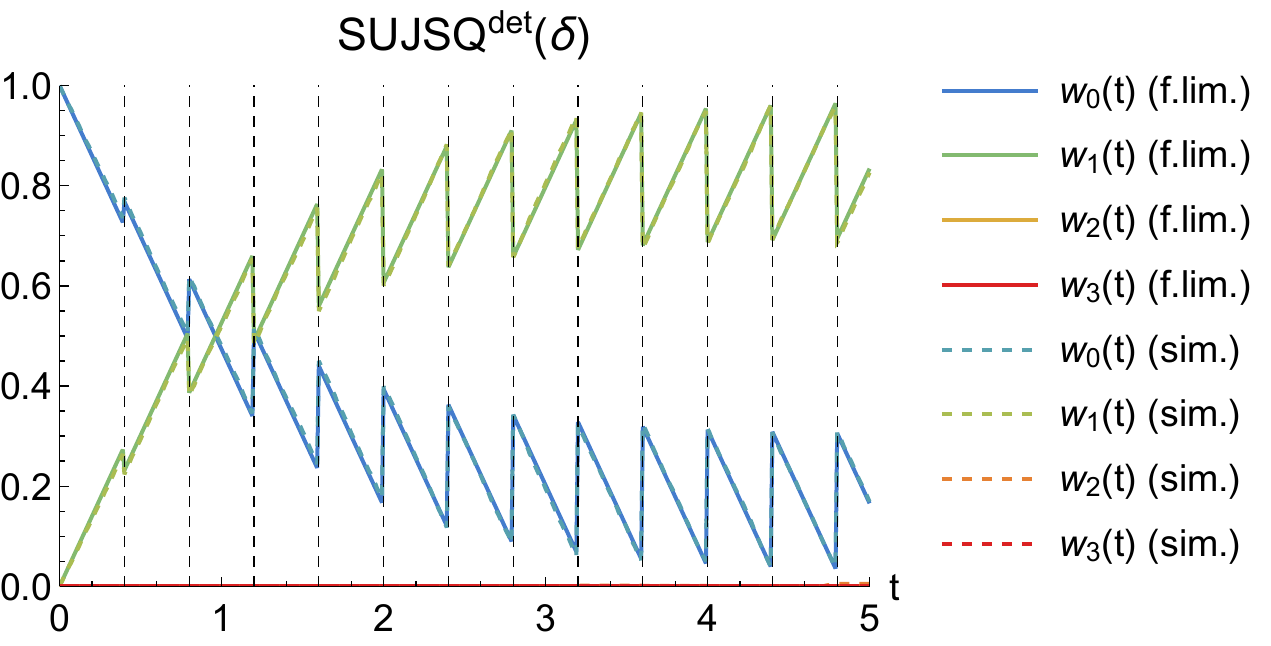} 
        \caption{Numerical emulation of the fluid limit for SUJSQ$^{\textup{det}}(2.5)$
        		and $\lambda = 0.7$, accompanied by simulation results for $N = 1000$,
        		averaged over 10~runs.}
        	\label{fig:ss13}
    \end{minipage}\hfill
    \begin{minipage}{0.49\textwidth}
        \centering
        \includegraphics[width=\textwidth]{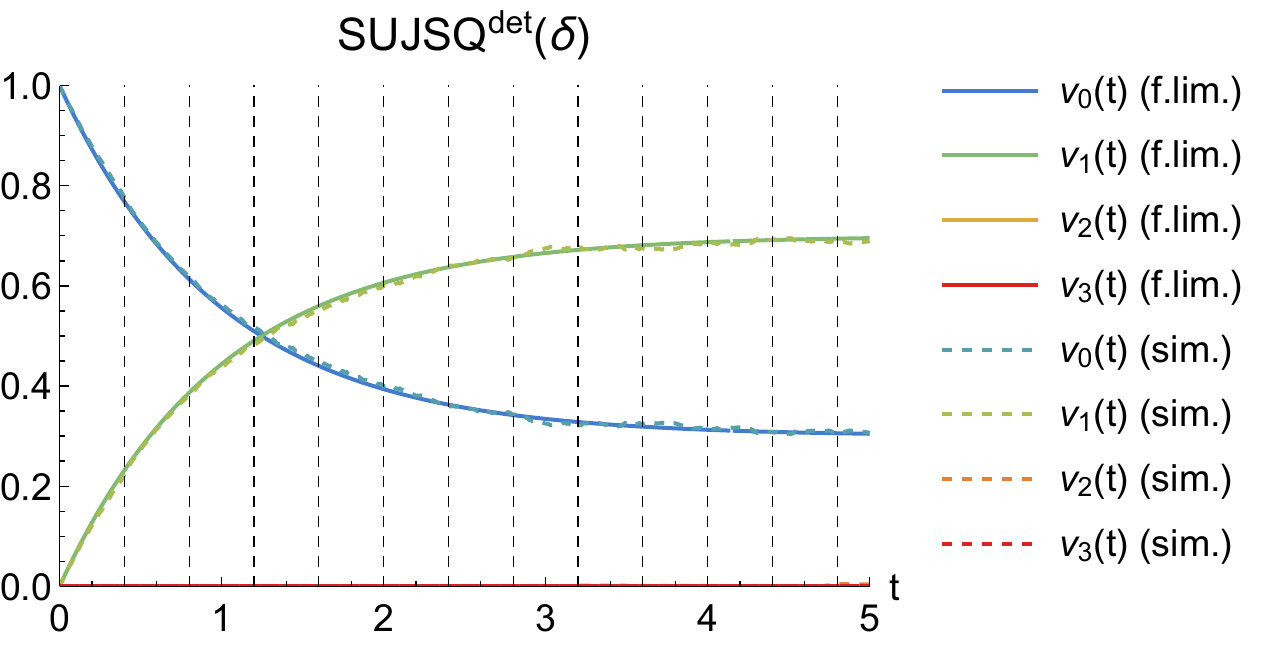} 
        \caption{Numerical emulation of the fluid limit for SUJSQ$^{\textup{det}}(2.5)$
        		and $\lambda = 0.7$, accompanied by simulation results for $N = 1000$,
        		averaged over 10~runs.}
        	\label{fig:ss14}
    \end{minipage}
\end{figure}

Further define $z_k(t) = \sum_{i = k}^{\infty} v_i(t)$ as the fraction
of servers with queue length~$k$ or larger at time~$t$ on fluid scale,
and note that $Q(t) = \sum_{k = 1}^{\infty} z_k(t)$.
We will also refer to $Q(t)$ as the total queue `mass' on fluid scale
at time~$t$, and introduce
\[
\begin{split}
Q^{\leq K}(t) &= \sum_{k = 1}^{K} z_k(t) =
\sum_{k = 1}^{K} \sum_{i = k}^{\infty} v_i(t) \\
&=
\sum_{i = 1}^{K} \sum_{k = 1}^{i} v_i(t) +
\sum_{i = K + 1}^{\infty} \sum_{k = 1}^{K} v_i(t) \\
&=
\sum_{i = 1}^{K} i v_i(t) + \sum_{i = K + 1}^{\infty} K v_i(t) =
\sum_{i = 1}^{\infty} \min\{i, K\} v_i(t),
\end{split}
\]
and
\[
\begin{split}
Q^{> K}(t) &= Q(t) - Q^{\leq K}(t) = \sum_{k = K + 1}^{\infty} z_k(t) =
\sum_{k = K + 1}^{\infty} \sum_{i = k}^{\infty} v_i(t) \\
&=
\sum_{i = K + 1}^{\infty} \sum_{k = K + 1}^{i} v_i(t) =
\sum_{i = K + 1}^{\infty} (i - K) v_i(t)
\end{split}
\]
as the queue mass (weakly) below and (strictly) above level~$K$,
respectively.

The fluid-limit equation~\eqref{eq:fl} yields the following expression
for the derivative of $z_k(t)$ (in between updates),
\[
\begin{split}
\frac{d z_k(t)}{dt} =&
\frac{d}{dt} \sum_{i = k}^{\infty} v_i(t) =
\frac{d}{dt} \sum_{i = k}^{\infty} \sum_{j = i}^{\infty} y_{i,j}(t) =
\sum_{i = k}^{\infty} \sum_{j = i}^{\infty} \frac{d y_{i,j}(t)}{dt} \\
=&\sum_{i = k}^{\infty} \sum_{j = i}^{\infty}
\big[y_{i+1,j}(t) \mathds{1}\{i < j\} - y_{i,j}(t) \mathds{1}\{i>0\} \\
&+
\lambda p_{i-1,j-1}(t) \mathds{1}\{i > 0\} - \lambda p_{i,j}(t) \big]\\
=&
\sum_{i = k}^{\infty} \big[v_{i+1}(t) - v_i(t) \mathds{1}\{i > 0\}\\
&+
\lambda \sum_{j = i}^{\infty} p_{i-1,j-1}(t) \mathds{1}\{i > 0\} -
\lambda \sum_{j = i}^{\infty} p_{i,j}(t)\big] \\
=&
- v_k(t) \mathds{1}\{k > 0\} + \lambda \sum_{j = k - 1}^{\infty}
p_{k-1,j}(t) \mathds{1}\{k > 0\},
\end{split}
\]
and
\begin{equation}
\begin{split}
\frac{d Q^{> K}(t)}{dt} =& \frac{d}{dt} \sum_{k = K + 1}^{\infty} z_k(t) =
\sum_{k = K + 1}^{\infty} \frac{d z_k(t)}{dt} \\
=&
- \sum_{k = K + 1}^{\infty} v_k(t) \mathds{1}\{k > 0\} \\
&+
\lambda \sum_{k = K + 1}^{\infty} \sum_{j = k - 1}^{\infty}
p_{k-1,j}(t) \mathds{1}\{k > 0\} \\
=&
- z_{K + 1}(t) + \lambda \sum_{k = K}^{\infty} \sum_{j = k}^{\infty}
\frac{y_{k,j}(t)}{w_j(t)} \mathds{1}\{m(t) = j\}.
\end{split}
\label{derivativequeuemass1}
\end{equation}
This may be interpreted by noting that the queue mass above level~$K$
increases due to arriving jobs being assigned to servers with queue
length~$K$ or larger and decreases due to jobs being completed
by servers with queue length $K + 1$ or larger.
In particular, we find that
\begin{equation}
\frac{d Q^{> K}(t)}{dt} = - z_{K + 1}(t)
\label{derivativequeuemass2}
\end{equation}
for all $K \geq m(t) + 1$.

Taking $K = 0$ and noting that $Q^{> 0}(t) \equiv Q(t)$, we obtain
\[
\frac{d Q(t)}{dt} = \lambda - z_1(t) = \lambda - [1 - v_0(t)],
\]
and thus
\begin{equation}
\label{eq:dyn}
Q(T) = Q(0) + \lambda T - \int_0^T [1 - v_0(t)] dt.
\end{equation}
This reflects that the average queue length on fluid scale at time~$T$
is obtained by adding the number of arrivals $\lambda T$ during
$[0, T]$ and subtracting the number of service completions, which
corresponds to the cumulative fraction of busy servers $1 - v_0(t)$. \\

The next lemma follows directly from~\eqref{derivativequeuemass2}, and shows that the queue mass above level~$K$ decreases
when the minimum queue estimate on fluid scale is strictly below~$K$,
so that there are no arrivals to servers with queue length~$K$ or larger.

\begin{lemma}

\label{basic1}

If $m(s) \leq K - 1$ for all $s \in [0, t)$,
then $Q^{> K}(t) \leq Q^{> K}(0)$.

\end{lemma}

We now proceed to derive a specific characterization of the decline
in the queue mass above level~$K$ when the minimum queue estimate
on fluid scale is strictly below~$K$.

For conciseness, denote
\[
A(L, t) = \sum_{l = 0}^{L} (L - l) \alpha_l(t)
\]
and
\newcommand{\inblocksP}{\sum_{l = 0}^{L} l \alpha_l(t) +
L \sum_{l = L + 1}^{\infty} \alpha_l(t)}
\[
B(L, t) := \inblocksP
\]
with $\alpha_l(t) = \frac{t^l}{l!} \ee^{- t}$,
and let $G$ be a Poisson random variable with parameter~$t$.
Note that $A(L, t) = {\mathbb E}[\max\{L - G, 0\}]$
and $B(L, t) = {\mathbb E}[\min\{G, L\}]$ so that $A(L, t) + B(L, t) = L$,
and in particular $B(L, t) \leq L$.
Observe that $A(L, T)$ may be interpreted as the expected queue length
after a time interval of length~$T$ at a single server with initial
queue length~$L$, unit-exponential service times and no arrivals,
while $B(L, T)$ may be interpreted as the expected number of service
completions during that time period.

\begin{lemma}

\label{lem:fund}

For any $K \geq 0$, $L \geq 1$, $t \geq 0$,
\[
\begin{split}
\int_0^t z_{K + 1}(s) ds
&\geq \left[Q^{\leq K + L}(0) - Q^{\leq K}(0)\right]
\left[1 - \frac{A(L, t)}{L}\right] \\
&= \frac{1}{L} \left[Q^{\leq K + L}(0) - Q^{\leq K}(0)\right] B(L, T).
\end{split}
\]

\end{lemma}

The proof of Lemma~\ref{lem:fund} involves a detailed analysis
of $z_{K+1}(t)$.
In the lemma special attention goes to the mass of jobs that are
queued in positions $K+1$ up to $K+L$,
represented by $Q^{\leq K + L}(0) - Q^{\leq K}(0)$.
The decline in this mass is no less than the decline in a situation where
the same total number of jobs reside with servers have either~$0$ jobs or
exactly $L$~jobs (so a fraction $[Q^{\leq K + L}(0) - Q^{\leq K}(0)] / L$
of the servers will have $L$~jobs).
Finally, $A(L, t)$ represents the expected number of jobs that remain
at time~$t$ at each of the servers with $L$~jobs, while $B(L, t)$
represents the expected number of jobs that have been completed
at time~$t$ by each of these servers.
These observations will be made rigorous in the proof
in Appendix~\ref{app:lemfund}.\\

The next lemma follows directly
from~\eqref{derivativequeuemass1} and~\eqref{derivativequeuemass2}
in conjunction with Lemma~\ref{lem:fund}. 

\begin{lemma}

\label{basic2}

For any $K \geq 0$, $L \geq 1$, $t \geq 0$,
\[
Q^{> K}(t) \leq \lambda t + Q^{> K + L}(0) +
\frac{1}{L} \left[Q^{\leq K + L}(0) - Q^{\leq K}(0)\right] A(L, t),
\]
or equivalently,
\[
Q^{> K}(t) - Q^{> K}(0) \leq \lambda t -
\frac{1}{L} \left[Q^{\leq K + L}(0) - Q^{\leq K}(0)\right] B(L, t).
\]

If $m(y(s)) \leq K - 1$ for all $s \in [0, t)$, then for any $L \geq 1$
\[
Q^{> K}(t) \leq Q^{> K + L}(0) +
\frac{1}{L} \left[Q^{\leq K + L}(0) - Q^{\leq K}(0)\right] A(L, t),
\]
or equivalently,
\[
Q^{> K}(t) - Q^{> K}(0) \leq
- \frac{1}{L} \left[Q^{\leq K + L}(0) - Q^{\leq K}(0)\right] B(L, t).
\]
In particular, taking $L = 1$,
\[
\begin{split}
Q^{> K}(t)
&\leq
Q^{> K + 1}(0) + \left[Q^{\leq K + 1}(0) - Q^{\leq K}(0)\right] \ee^{- t} \\
&= Q^{> K + 1}(0) + z_{K + 1}(0) \ee^{- t},
\end{split}
\]
yielding
\begin{equation}
Q^{> K}(t) \leq
\ee^{- t} Q^{> K}(0) + [1 - \ee^{- t}] Q^{> K + 1}(0).
\end{equation}
\end{lemma}

The next lemma provides a simple condition for the minimum queue
estimate on fluid scale to remain strictly below~$K$ throughout the
interval $[0, t)$ in terms of $Q(0)$ and the proof is provided in Appendix \ref{app:lemms}

\begin{lemma}

\label{auxiliary}

If $Q(0) \leq K - \lambda t$, then $m(s) \leq K - 1$ for all
$s \in [0, t)$.

\end{lemma}

We will henceforth say that $L$ is large enough if
\begin{equation}
\lambda T < \sigma(L; \lambda, T) =
\left(1 - \frac{\lambda T + 1}{L}\right) B(L, T),
\label{eq:cond}
\end{equation}
and define $s(\lambda,T) = \min\{L: \lambda T < \sigma(L; \lambda, T)\}$,
which may be loosely thought of as the maximum queue length on fluid
scale in the sense of the next proposition.
\newcommand{\maxq}{{s}}
Note that $\sigma(L; \lambda, T) \uparrow T$ as $L \to \infty$,
ensuring that $\maxq(\lambda, T)$ is finite for any $\lambda < 1$.

\begin{proposition}[Bounded queue length for $\sujsqd$]

\label{prop:sync2}

For any initial state $y(0)$ with finite queue mass $Q(0) < \infty$,
the fraction of servers on fluid scale with a queue length larger
than $s(\lambda,T)$ vanishes over time.
Additionally, if the initial queue mass $Q(0)$ is sufficiently small
and the initial fraction of servers with a queue length larger
than $s(\lambda,T)$ is zero, then that fraction will remain zero forever. 

\end{proposition}



The proof of Proposition~\ref{prop:sync2} leverages Lemma~\ref{lem:fund}
and is organized as follows.
For any initial state, we can show that either the mass in the tail,
or the total mass is decreasing.
Once one of the two is below a certain level, we show that the other
decreases as well.
We show this in two lemmas.
From that point on, it is a back and forth between decreasing mass in the
tail and decreasing total mass, which is described in the final lemma.
The mass in the tail will decrease, such that the mass strictly above
level $L^* = s(\lambda, T)$ will vanish.

Define 
\begin{equation}\label{eq:delta}
\Delta = \left(1 - \frac{\lambda T + 1}{L^*}\right) B(L^*, T) - \lambda T.
\end{equation}
We now state two corollaries, which provide upper bounds for the total
queue mass at time~$T$.
The proofs are based on Lemma~\ref{lem:fund} and provided
in Appendices~\ref{c1} and~\ref{c2}.

\begin{corollary}

\label{corollary1}

If $L$ is large enough and $Q(0) \geq L - 1 - \lambda T$, then
\[
\begin{split}
Q(T) &\leq Q(0) - \Delta + \frac{Q^{> L}(0)}{L} B(L, T) \\
&<
Q(0) - \Delta + Q^{> L}(0).
\end{split}
\]

\end{corollary}

\begin{corollary}

\label{corollary2}

If $L$ is large enough and $Q(0) \leq L - 1 - \lambda T$, then
\[
\begin{split}
Q(T) &\leq L - 1 - \lambda T - \Delta + \frac{Q^{> L}(0)}{L} B(L, T) \\
&<
L - 1 - \lambda T - \Delta + Q^{> L}(0).
\end{split}
\]

\end{corollary}

We now present Lemmas~\ref{auxiliary1} and~\ref{auxiliary2},
which use Corollaries~\ref{corollary1} and~\ref{corollary2}
and Lemmas~\ref{basic1}, \ref{basic2} and~\ref{auxiliary} to show that
under certain conditions, the total queue mass as well as the mass
above level~$L$ strictly decrease.

\begin{lemma}

\label{auxiliary1}

If $L$ is large enough, $L - 1 - \lambda T \leq Q(0) \leq L - \lambda T$
and $Q^{> L}(0) < \Delta$, then
\begin{itemize}
\item $Q(T) < Q(0) - D$, where $D > 0$, so $Q$ is strictly
smaller at the next update,
\item $Q^{> L}(T) \leq Q^{> L}(0)$, so $Q^{> L}$ remains strictly smaller
than $\Delta$.
\end{itemize}

\end{lemma}

\begin{proof}
The first statement follows from Corollary~\ref{corollary1}, with $D=\Delta/2 - Q^{>L}(0) / 2$.
The second assertion follows from Lemmas~\ref{basic1} and~\ref{auxiliary}.
\end{proof}

\begin{lemma}

\label{auxiliary2}

If $L$ is large enough, $Q(0) \leq L - 1 - \lambda T$
and $Q^{> L}(0) < \Delta$ then
\begin{itemize}
\item $Q(T) < L - 1 - \lambda T$, so $Q$ remains strictly smaller than
$L - 1 - \lambda T$ at the next update,
\item $Q^{> L}(T) \leq Q^{>L}(0)$, so $Q^{> L}$ remains strictly smaller
than $\Delta$,
\item $Q^{> L - 1}(T) \leq c Q^{> L-1}(0)$ if $Q^{>L-1}(0)\geq \Delta$, where $c < 1$,
so $Q^{> L-1}$ is strictly decreasing by a constant factor over each
update interval.
\end{itemize}

\end{lemma}

\begin{proof}
The first statement follows from Corollary~\ref{corollary2}.
The second assertion follows from Lemmas~\ref{basic1} and~\ref{auxiliary}.
The third statement holds for
\[
c = \ee^{-T} + (1-\ee^{-T}) \frac{Q^{>L}(0)}{\Delta} < 1
\]
since the final portion of Lemma \ref{basic2} in conjunction with Lemma \ref{auxiliary} gives 
\[
\begin{split}
\frac{Q^{>L-1}(T)}{Q^{>L-1}(0)} &\leq \ee^{-T} + (1-\ee^{-T}) \frac{Q^{>L}(0)}{Q^{>L-1}(0)} \\
&\leq \ee^{-T} + (1-\ee^{-T}) \frac{Q^{>L}(0)}{\Delta}.
\end{split}
\]

\end{proof}

\begin{lemma}
\label{lem:maxave}
If $L$ is large enough, $Q(0) \leq L - \lambda T$
and $Q^{> L}(0) < \Delta$, then there exists
a finite time $\tau_{L}^*$ such that
$Q(\tau_{L}^*) \leq L - 1 - \lambda T$
and $Q^{> L - 1}(\tau_{L}^*) < \Delta$ (as defined in \eqref{eq:delta}).
\end{lemma}

\begin{proof}
The proof is constructed by applying Lemmas~\ref{auxiliary1}
and~\ref{auxiliary2} in succession.
Since $Q(0) \leq L - \lambda T$ and $Q^{> L}(0) < \Delta$,
Lemma~\ref{auxiliary1} can be applied, so that $Q$ is strictly
decreasing and eventually becomes smaller than $L - 1 - \lambda T$ while $Q^{>L}$ remains smaller than $\Delta$. Note that $D$ does not decrease after any iteration since $Q^{>L}(0)$ can only decrease.
At that moment, Lemma~\ref{auxiliary2} can be applied which shows that
$Q^{> L - 1}$ decreases by a constant factor over each update interval
as long as $Q^{> L - 1}$ is larger than $\Delta$. The constant factor $c$ does not increase after any iteration, and in fact only becomes smaller as $Q^{>L}(0)$ decreases.
In other words, $Q^{> L - 1}$ becomes smaller than $\Delta$
after finitely many updates, while $Q$ remains smaller than $L-1-\lambda T$ and $Q^{>L}$ smaller than $\Delta$.
\end{proof}

\textit{Proof of Proposition~\ref{prop:sync2}.}
Observe that the procedure in the proof of Lemma~\ref{lem:maxave}
can be performed as long as $L$ is large enough.
The left-hand side of~(\ref{eq:cond}) is increasing in~$L$
(as both factors are increasing in~$L$),
which shows that the condition~(\ref{eq:cond}) becomes tighter
for smaller values of~$L$.
In fact, the mass above level~$L^*$ will vanish,
where $L^*$ is the lowest value of~$L$ which is sufficiently large
for (\ref{eq:cond}) to hold, yielding the first statement
of Proposition~\ref{prop:sync2}.  The latter part follows directly from Lemma \ref{auxiliary1}.

Finally, we stress that the lemma can also be applied when the maximum
initial queue length is infinite, but the mass $Q(0) = \sum_{i=0}^\infty z_i(0) < \infty$ is finite.
In that case, one can find a value for $\bar{L}$ such that
$Q^{>\bar{L}}(0) = \sum_{i=\bar{L}}^\infty z_i(0) < \Delta$ (as the tail of a convergent series tends to zero) and $Q(0) \leq \bar{L} - \lambda T$.
In that case, the lemma can be successively applied,
starting from $\bar{L}$, which proves Proposition~\ref{prop:sync2}.\\

We now proceed to state the second main result in this section.
Denote by $y^*$ the fluid state with $y_{0,0}^* = 1 - \lambda$,
$y_{0,1}^* = 0$, $y_{1,1}^* = \lambda$, and $y_{i,j}^* = 0$ for all
$j \geq i \geq 2$.

\begin{proposition}[No-queueing threshold for $\d$ in $\sujsqd$]
\label{prop:sync1}
Suppose $\delta > \lambda / (1 - \lambda)$, or equivalently,
$T = 1 / \delta < (1 - \lambda) / \lambda = 1 /\lambda - 1$.
Then $y^*$ is a fixed point of the fluid-limit process at update
moments in the following sense:
\begin{enumerate}
\item[(a)] If $y(0) = y^*$, then $y(k T) = y^*$ for all $k \geq 0$;
\item[(b)] For any initial state $y(0)$ with $Q(0) < \infty$,
$y(k T) \to y^*$ as $k \to \infty$.
\end{enumerate}
Moreover, in case~(a), $y_{0,0}(k T + t) = 1 - \lambda - \lambda t$,
$y_{0,1}(k T + t) = \lambda t$,
$y_{1,1}(k T + t) = \lambda$ for all $k \geq 0$, $t \in [0, T)$,
and $y_{i,j}(t) = 0$ for all $j \geq i \geq 2$, $t \geq 0$.\\
In case~(b), $y_{0,0}(k T + t) \to 1 - \lambda - \lambda t$,
$y_{0,1}(k T + t) \to \lambda t$,
$y_{1,1}(k T + t) \to \lambda$ for all $k \geq 0$, $t \in [0, T)$,
and $y_{i,j}(t) \to 0$ for all $j \geq i \geq 2$, $t \geq 0$.
\end{proposition}

Loosely speaking, Proposition~\ref{prop:sync1} implies that for $\d \geq \lambda/(1-\lambda)$, in the long term the fraction of jobs that incur a non-zero waiting time vanishes.
We note that in this regime, jobs are only sent to idle servers, which means that servers only need to send feedback whenever they are idle at an update moment. Since the fraction of idle servers in the fixed point is $1-\lambda$, a sparsified version of $\sujsqd$ will have a communication overhead of $\lambda$ per time unit or 1 message per job.

The next lemma, whose proof is presented in Appendix~\ref{app:lembounds},
provides lower and upper bounds for the number of service completions
on fluid level and the total queue mass, which play an instrumental
role in the proof of Proposition~\ref{prop:sync1}.
The bounds and proof arguments are similar in spirit to those
of Lemma~\ref{lem:fund} with $K = 0$, but involve crucial refinements
by additionally accounting for service completions of arriving jobs.

\begin{lemma}

\label{lem:bounds}

If $v_0(0) \leq 1 - \lambda$, then (i) $v_0(T) \leq 1 - \lambda$. 

If $v_0(0) \geq \lambda T$, then $\int_0^T [1 -v_0(t)] dt$ is bounded
from below by
\[
\begin{split}
& \lambda [T - 1 + \ee^{- T}] +
z_1(0) [1 - \ee^{- T}] + z_2(0) [1 - \ee^{- T} - T \ee^{- T}] \\
=
& \lambda T + [z_1(0) - \lambda] [1 - \ee^{- T}] +
z_2(0) [1 - \ee^{- T} - T \ee^{- T}],
\end{split}
\]
so that in view of~\eqref{eq:dyn} (ii)
\[
Q(T) \leq Q(0) - [z_1(0) - \lambda] [1 - \ee^{- T}] -
z_2(0) [1 - \ee^{- T} - T \ee^{- T}],
\]
and
\[
Q(T) \geq Q(0) - [z_1(0) - \lambda] [1 - \ee^{- T}] - z_2(0) T,
\]
so that in view of~\eqref{eq:dyn} (iii),
\[
\begin{split}
\int_0^T [1 - v_0(t)] dt &\leq \lambda [T - 1 + \ee^{- T}] +
z_1(0) [1 - \ee^{- T}] + z_2(0) T \\
&=
\lambda T + [z_1(0) - \lambda] [1 - \ee^{- T}] + z_2(0) T.
\end{split}
\]

If $v_0(0) < \lambda T$, then $\int_0^T [1 - v_0(t)] dt$ is bounded
from below by
\[
\begin{split}
& \lambda [T - 1 + \ee^{- T}] + \hat{z}_1(0) [1 - \ee^{- T}] +
\hat{z}_2(0) [1 - \ee^{- T} - T \ee^{- T}] \\
=
& \lambda T + [\hat{z}_1(0) - \lambda] [1 - \ee^{- T}] +
\hat{z}_2(0) [1 - \ee^{- T} - T \ee^{- T}],
\end{split}
\]
so that in view of~\eqref{eq:dyn} (iv)
\[
Q(T) \leq
Q(0) - [\hat{z}_1(0) - \lambda] [1 - \ee^{- T}] -
\hat{z}_2(0) [1 - \ee^{- T} - T \ee^{- T}],
\]
with $\hat{z}_i(0) = \min\{z_i(0), 1 - \lambda T\}$, $i = 1, 2$.

We deduce that (v)
\[
\begin{split}
Q(T) \leq Q(0) &- \min\{1 - v_0(0) - \lambda, 1-\lambda T - \lambda\} [1 - \ee^{- T}] \\
&- \min\{z_2(0), 1 - \lambda T\} [1 - \ee^{- T} - T \ee^{- T}],
\end{split}
\]
with $1 - \lambda T - \lambda > 0$.

\end{lemma}

\begin{proof}[Proof of Proposition~\ref{prop:sync1}]

We first consider case~(a) with $y(0) = y^*$.
The fluid-limit equations can then be explicitly solved to obtain
$y_{0,0}(t) = 1 - \lambda - \lambda t$, $y_{0,1}(t) = \lambda t$,
$y_{1,1}(t) = \lambda$, and $y_{i,j}(t) = 0$ for all $j \geq i \geq 2$
for $t \in [0,T)$.

Since at an update moment
$y_{0,0}(T) = y_{0,0}(T^-) + y_{0,1}(T^-) = 1 - \lambda$,
$y_{1,1}(T) = y_{1,1}(T^-) = \lambda$,
and $y_{i,j}(t) = y_{i,j}(T^-) = 0$ for all $j \geq i \geq 2$,
we obtain a strictly cyclic evolution pattern with $y(k T) = y^*$
for all $k \geq 0$, as well as
$y_{0,0}(k T + t) = 1 - \lambda - \lambda t$,
$y_{0,1}(k T + t) = \lambda t$,
$y_{1,1}(k T + t) = \lambda$ for all $k \geq 0$ $t \in [0, T)$,
and $y_{i,j}(t) = 0$ for all $j \geq i \geq 2$, $t \geq 0$.

We now turn to case~(b).
First suppose that there exists a $k_0 < \infty$ such that
$v_0(k_0 T) \leq 1 - \lambda$.
It then follows from statement~(i) in Lemma~\ref{lem:bounds} that
$v_0(k T) \leq 1 - \lambda$ for all $k \geq k_0$.
Moreover, in view of statement~(v) in Lemma~\ref{lem:bounds} we have $Q((k + 1) T) \leq Q(k T) - c \min\{\epsilon, \Delta\} -
d \min\{z_2(k T), 1 - \lambda T\}$, with $c > 0$ and $d > 0$ when
$v_0(k T) \leq 1 - \lambda - \epsilon$ for any $\epsilon > 0$.
Thus, for any $\epsilon > 0$ it can only occur finitely many times
that $v_0(k T) \leq 1 - \lambda - \epsilon$, and additionally for any
$\delta > 0$, it can only occur finitely many times that
$z_2(k T) \geq \delta$, because otherwise $Q(k T)$ would eventually
fall to zero, which would contradict $v_0(k T) \leq 1 - \lambda$.
Thus we conclude that $v_0(k T) \to 1 - \lambda$ and $z_2(k T) \to 0$,
which implies that $y(k T) \to y^*$ as $k \to \infty$ as stated.

Now suppose that there exists no $k_0 < \infty$ such that
$v_0(k_0 T) \leq 1 - \lambda$, i.e, $v_0(k T) > 1 - \lambda$
for all $k \geq 1$.
Solving \eqref{eq:fl} then gives $\frac{d}{dt} w_0(t) = - \frac{d}{dt} w_1(t) = -
\lambda$ and  $m(t) = 0$ for $t \in [k T, k T + v_0(k T)/\lambda]$, where $v_0(k T)/\lambda > T$.
Lemma~\ref{basic2} with $K = 1$, $L = 1$ yields
\[
Q^{> 1}((k + 1) T) \leq Q^{> 1}(k T) - z_2(k T) [1 - \ee^{- T}].
\]
Thus we must have $z_2(k T) \to 0$ as $k \to \infty$,
because otherwise $Q^{> 1}(k T)$ would eventually drop below zero,
which would contradict the fact that it must always be positive.
Hence, for any $\epsilon > 0$, there exists $k_\epsilon$ such that $z_2(k T) \leq \epsilon (1-\ee^{- T}) / (2T)$ for all $k\geq k_\epsilon$,.
Statement~(iii) in Lemma~\ref{lem:bounds} may then be invoked
to obtain that for any $\epsilon > 0$ and $k \geq k_\epsilon$ if
$v_0(k T) \geq 1 - \lambda + \epsilon \geq \lambda T$, then
\[
Q((k + 1) T) \geq Q(k T) + \frac{\epsilon}{2} \frac{1-\ee^{-T}}{T}.
\]
It follows that for any $\epsilon > 0$, it can only occur finitely
many times that $v_0(k T) \geq 1 - \lambda + \epsilon$, as $Q(kT)$ is bounded since $Q^{>1}(kT)$ is decreasing (Lemma \ref{basic1}).
Thus we conclude $v_0(k T) \to 1 - \lambda$ and $z_2(k T) \to 0$
as $k \to \infty$, implying that $y(k T) \to y^*$ as $k \to \infty$
as stated.

\end{proof}

\section{Asynchronous updates}
\label{sec:async}

In this section we turn to the fluid limit for asynchronous updates.
As mentioned earlier, for non-exponential update intervals,
the state variables $Y_{ij}$ would need to be augmented with continuous
variables in order to obtain a Markovian state description.
This would give rise to a measure-valued fluid-limit description,
and involve heavy technical machinery, but provide little insight,
and hence we will focus on the fluid limit for exponential update intervals.
In Subsection~\ref{subsec:dynaasync} we provide a characterization
of the fluid-limit trajectory, along with a heuristic explanation,
numerical illustration and comparison with simulation.
In Subsection~\ref{subsec:fpasync} the fixed point of the fluid limit
is determined (Proposition~\ref{fp}), which immediately shows
that in stationarity queueing vanishes at fluid level for sufficiently
high~$\d$ (Corollary~\ref{cor:async1}) and also provides an upper
bound for the queue length at fluid level for any given $\d > 0$
(Corollary~\ref{coro:asyn1}).


\subsection{Fluid-limit dynamics}
\label{subsec:dynaasync}

In the case of synchronous updates, the minimum queue estimate
$m(y^N(t))$ could never decrease between successive update moments.
As a result, the amount of time $\int_{t_0}^{t} \indis{m(Y^N(s)) = j} ds$
that the minimum queue length equals~$j$ in between successive
update moments converges to $\int_{t_0}^{t} \alpha_j(s) ds$,
as $N \to \infty$, with $\alpha_j(t) = \indis{m(y(t)) = j}$ and can be
directly expressed in terms of the minimum queue estimate on fluid scale.

In contrast, with asynchronous updates, the minimum queue estimate
may drop at any time when an individual server with a queue length
$i < m(y^N(t))$ sends an update at time~$t$, and becomes the only server
with a queue estimate below $m(y^N(t))$.
Consequently, the amount of time $\int_{t_0}^{t} \indis{m(Y^N(s)) = j} ds$
that the minimum queue length equals~$j$ no longer tends to
$\int_{t_0}^{t} \alpha_j(s) ds$ as $N \to \infty$, and may even have
a positive derivative for $j < m(y(t))$, i.e., for queue values
strictly smaller than the minimum queue estimate on fluid scale.

The fact that even in the limit the system may spend a non-negligible
amount of time in states that are not directly visible on fluid scale
severely complicates the characterization of the fluid limit.
In order to handle this complication and describe the evolution of the
fluid limit, it is convenient to define
$u_k(t) = \delta \sum_{i = 0}^{k - 1} (k - i) v_i(t)$
as the fluid-scaled rate at which the dispatcher can assign jobs
to servers with queue estimates below~$k$ as a result of updates,
with $v_i(t) = \sum_{l = i}^{\infty} y_{i,l}(t)$ representing the fraction
of servers with queue length~$i$ in fluid state~$y$ at time $t$ as before.

We distinguish two cases, depending on whether
$u_{m(t)}(t)$  $\leq \lambda$ or not, and additionally introduce
$n(t)$, defined as $n(t) = m(t)$ in case
$u_{m(t)}(t) \leq \lambda$,
or $n(t) = \min\{n: u_n(t) > \lambda\} \leq m(t) - 1$ otherwise.
Then servers with a true queue length $i \leq n(t) - 1$ will be
assigned $n(t) - i$ jobs almost immediately after an update at
time~$t$, and then have both queue length and queue estimate $n(t)$.
Incoming jobs will be assigned to servers with a queue estimate
at most $n(t) - 1$ at rate $u_{n(t)}(t)$ and to servers
with queue estimate exactly equal to $n(t)$ at rate
$\zeta(t) = \lambda - u_{n(t)}(t)$.

Then the fluid limit $y(t)$ satisfies the system of differential equations
\begin{align}
\begin{split}
\label{eq:flasyn}
\frac{dy_{i,j}(t)}{dt} &=
y_{i+1,j}(t) \mathds{1}\{i < j\} - y_{i,j}(t) \mathds{1}\{i>0\} \\
& + \zeta(t) q_{i-1,j-1}(t) \mathds{1}\{i > 0\} - \zeta(t) q_{i,j}(t) \\
& +
\delta \sum_{k = 0}^{n(t) - 1} v_k(t) \mathds{1}\{i = j = n(t)\} \\
&+
\delta v_i(t) \mathds{1}\{i = j \geq n(t)\} - \delta y_{ij}(t),
\end{split}
\end{align}
for all $i = 0, 1, \dots, j \geq n(t)$, where
\[
q_{i,j}(t) = \frac{y_{i,j}(t)}{w_j(t)} \mathds{1}\{n(t)=j\}
\]
denotes the fraction of jobs assigned to a server with queue
length~$i$ and queue estimate~$j$ in fluid state~$y$ among the ones
that are assigned to a server with queue estimate at least $n(t)$,
defined as function of the fluid state~$y$ at time $t$ as above.

It can be checked that when $n(t) < m(t)$, the derivative
of $\sum_{j = 0}^{m(t)} w_j(t)$ is strictly positive, i.e., the
fraction of servers with a queue estimate below $m(t)$ becomes
positive, and the value of $m(t)$ instantly becomes equal to $n(t)$.

An informal outline of the derivation of the fluid limit as stated
in~\eqref{eq:flasyn} is provided in Appendix~\ref{sketchasyn}.

\subsubsection{Interpretation}

The above system of differential equations may be intuitively
interpreted as follows.
The first two terms correspond to service completions at servers
with $i + 1$ and $i$~jobs, just like in~\eqref{eq:fl}.
The third and fourth terms account for job assignments to servers
with a queue estimate $m(t)$.
The third term captures the resulting increase in the fraction
of servers with queue estimate $m(t) + 1$, while the fourth
term captures the corresponding decrease in the fraction
of servers with queue estimate $m(t)$.

The final three terms in~\eqref{eq:flasyn} correspond to the updates
from servers received at a rate~$\delta$.
The fifth term represents the increase in the fraction of servers
with queue estimate $n(t)$ due to updates from servers with
a queue length $k \leq n(t) - 1$ which are almost immediately being
assigned $n(t) - k$ jobs and then have both queue length
$i = n(t)$ and queue estimate $j = n(t)$.
The sixth term represents the increase in the fraction of servers
with a queue estimate $j = n(t)$ or larger due to updates
from servers with a queue length $i = j$.
The final term represents the decrease in the number of servers
with queue length~$i$ and queue estimate~$j$ due to updates.

Even though a non-zero fraction of the jobs are assigned to servers
with a queue estimate below $n(t)$, these events are not directly
visible at fluid level, and only implicitly enter the fluid limit
through the thinned arrival rate $\zeta(t)$.

Summing the equations~\eqref{eq:flasyn} over $i = 0, 1, \dots, j$ yields
\begin{align*}
\frac{dw_j(t)}{dt} &=
\zeta(t) [\mathds{1}\{n(t)=j-1\} - \mathds{1}\{n(t)=j\}] \\
&+
\delta [v_j(t) - w_j(t)] \mathds{1}\{j \geq n(t)\},
\end{align*}
reflecting that servers with queue estimate $n(t)$ are assigned jobs,
and thus flipped into servers with queue estimate $n(t) + 1$,
at rate~$\zeta(t)$, and that servers with a queue estimate
$j \geq n(t)$ are created at an effective rate
$\delta [v_j(t) - w_j(t)]$ as a result of updates.

\begin{figure}
    \centering
    \begin{minipage}{0.49\textwidth}
        \centering
        \includegraphics[width=\textwidth]{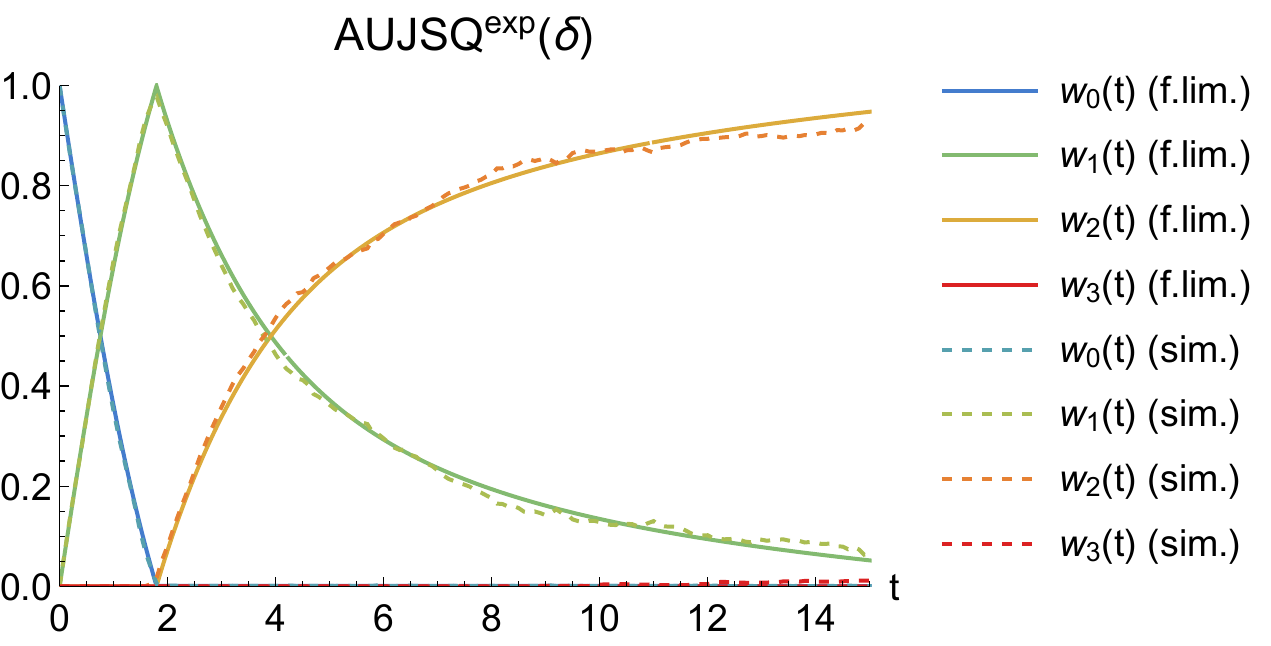} 
        \caption{Numerical emulation of the fluid limit for AUJSQ$^{\textup{exp}}(0.85)$
        		and $\lambda = 0.7$, accompanied by simulation results with $N = 1000$,
        		averaged over 10~runs.}\label{fig:ss31}
    \end{minipage}\hfill
    \begin{minipage}{0.49\textwidth}
        \centering
        \includegraphics[width=\textwidth]{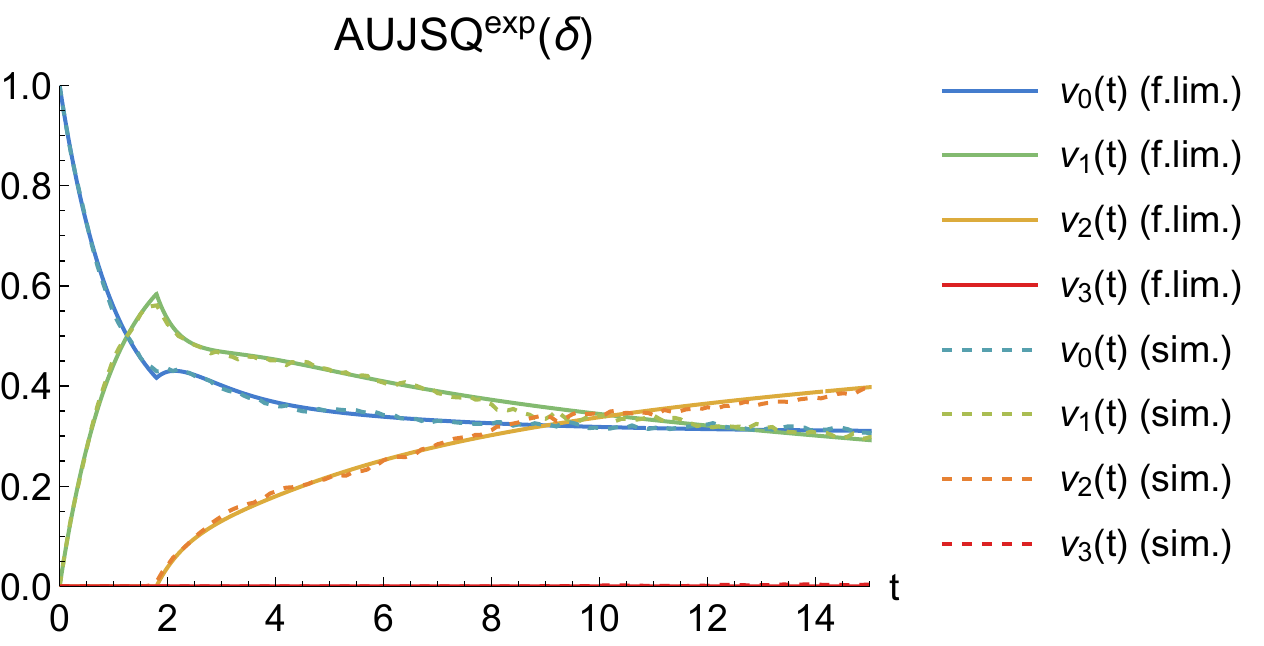} 
        \caption{Numerical emulation of the fluid limit for AUJSQ$^{\textup{exp}}(0.85)$
        		and $\lambda = 0.7$, accompanied by simulation results with $N = 1000$,
        		averaged over 10~runs.}\label{fig:ss32}
    \end{minipage}
\end{figure}

\subsubsection{Numerical illustration and comparison with simulation}
\label{simuasyn}

Figures~\ref{fig:ss31}-\ref{fig:ss34} show the fluid-limit
trajectories~$y(t)$ as governed by the differential equations
in~\eqref{eq:flasyn} for $\aujsqe$, through stochastic simulation
for a system with $N = 1000$ servers and averaged over 10~runs.
Once again, the simulation results are nearly indistinguishable
from the fluid-limit trajectories.

In contrast to the synchronous variants
in Figures~\ref{fig:ss11}--\ref{fig:ss14}, the trajectories do not
oscillate, but approach stable values, corresponding to the fixed
point of the fluid-limit equations~\eqref{eq:flasyn} which we will
analytically determine in Proposition~\ref{fp}.
In Figures~\ref{fig:ss11} and~\ref{fig:ss12} where $\delta = 0.85$
is relatively low, we observe once again that $w_2(y(t))=w_2(t)$ and $v_2(y(t))=v_2(t)$
become strictly positive.
In Figures~\ref{fig:ss13} and~\ref{fig:ss14} where $\delta = 2.5$ is
sufficiently large, all servers have either zero or one jobs in the limit,
indicating that no queueing occurs.

Qualitatively similar results are observed for $\aujsqd$,
where the updates occur at strictly regular moments.
The results are displayed in Figures~\ref{fig:ss41}--\ref{fig:ss44},
which are again relegated to Appendix~\ref{app:simresasyn} because
of space limitations.

\begin{figure}
    \centering
    \begin{minipage}{0.49\textwidth}
        \centering
        \includegraphics[width=\textwidth]{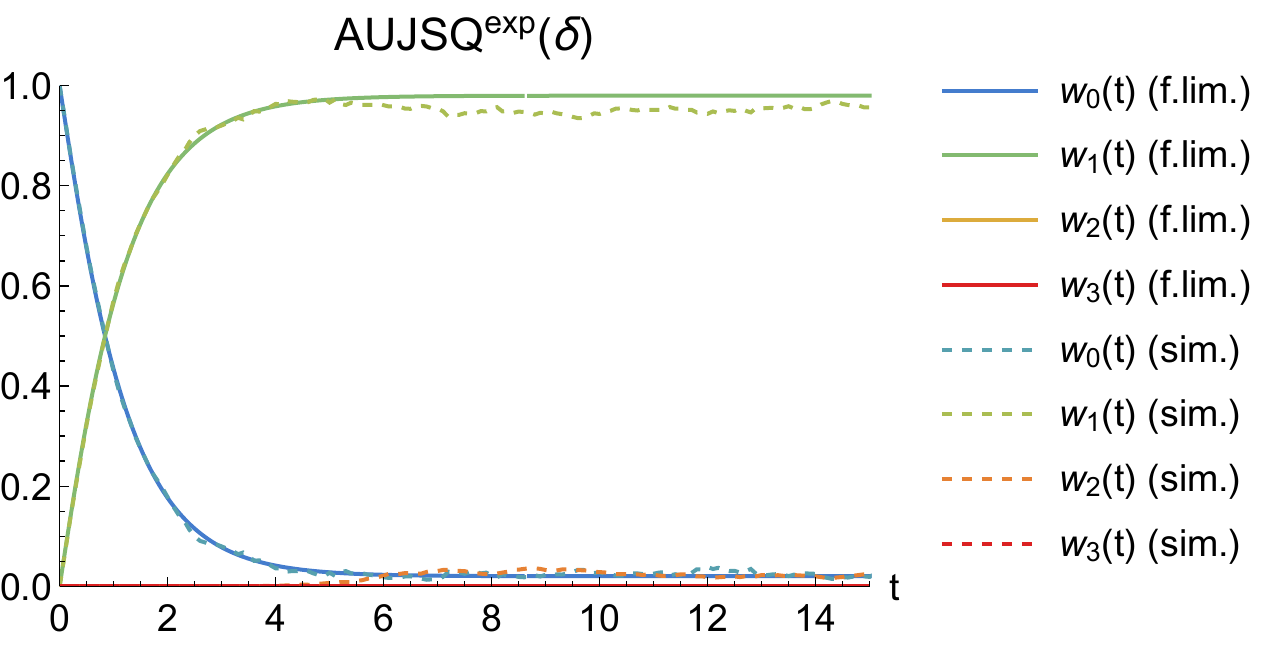} 
        	\caption{Numerical emulation of the fluid limit for AUJSQ$^{\textup{exp}}(2.5)$
        		and $\lambda = 0.7$, accompanied by simulation results with $N = 1000$,
        		averaged over 10~runs.}\label{fig:ss33}
    \end{minipage}\hfill
    \begin{minipage}{0.49\textwidth}
        \centering
        \includegraphics[width=\textwidth]{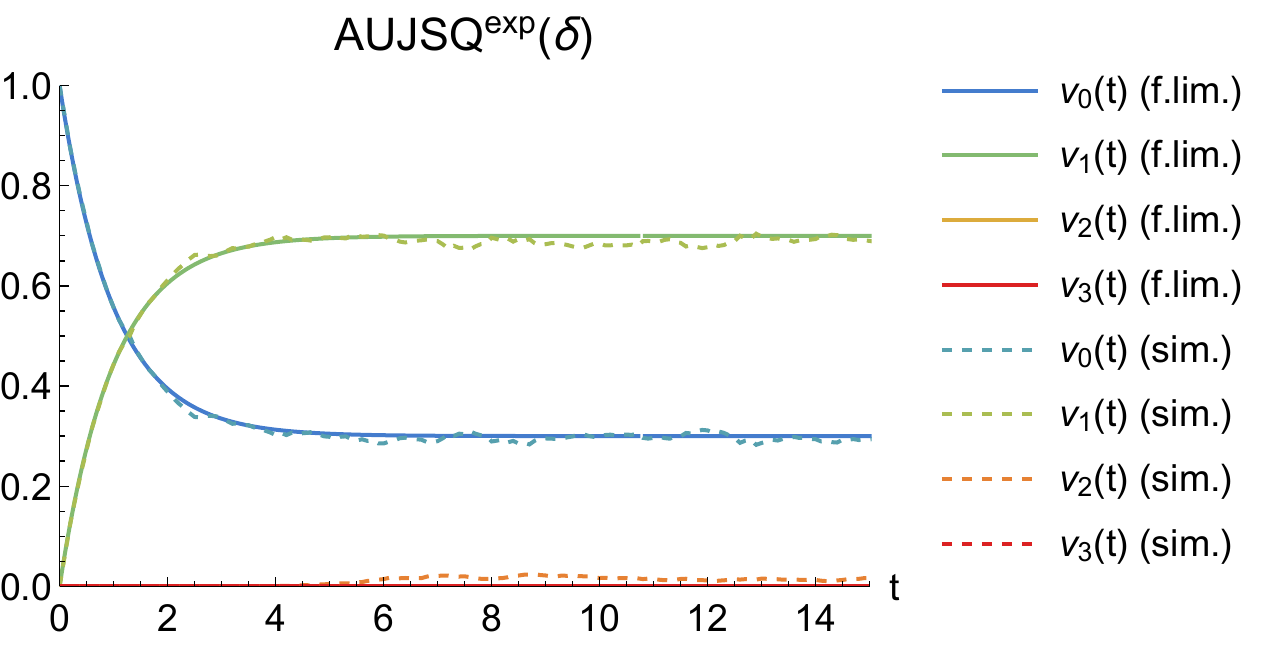} 
        \caption{Numerical emulation of the fluid limit for AUJSQ$^{\textup{exp}}(2.5)$
        		and $\lambda = 0.7$, accompanied by simulation results with $N = 1000$,
        		averaged over 10~runs.}\label{fig:ss34}
    \end{minipage}
\end{figure}

\subsection{Fixed-point analysis}
\label{subsec:fpasync}

The next proposition identifies the fixed point of the fluid-limit
equations~\eqref{eq:flasyn} in terms of~$m^*$, defined as
\begin{equation}
\label{eq:mstarexp}
\begin{split}
m^* &= m(\lambda, \delta) =
\min\left\{m: \lambda < 1 - \left(\frac{1}{1 + \delta}\right)^{m + 1}\right\}\\
&=
\max\left\{m: \lambda \geq 1 - \left(\frac{1}{1 + \delta}\right)^m\right\} =
\left\lfloor - \frac{\log(1 - \lambda)}{\log(1 + \delta)} \right\rfloor,
\end{split}
\end{equation}
which may be interpreted as the minimum queue estimate at fluid level
in stationarity.
For compactness, define $a = \frac{1}{1 + \delta}$ and $b = \frac{1}{1  + \delta + \nu}$.

\begin{proposition}[Fixed point for $\aujsqe$]

\label{fp}

The fixed point of the fluid limit \eqref{eq:flasyn} is given by

\begin{eqnarray*}
y_{0,m^*}^*
&=&
\frac{ab^{m^* - 1}\d}{(1+\nu)(\delta+\nu)},  \\
y_{i,m^*}^*
&=&
\frac{ab^{m^* - i}\d}{1+\nu},
\hspace*{.2in} i = 1, \dots, m^*,  \\
y_{0,m^*+1}^*
&=&
a^{m^*+1}-\frac{a^2 b^{m^*-1}\d}{(1+\nu)(\d+\nu)}, \\
y_{1,m^*+1}^*
&=&
\d \left(a^{m^*+1}-\frac{a^2 b^{m^*-1}\d}{(1+\nu)(\d+\nu)}\right), \\
y_{i,m^*+1}^*
&=&
\d\left(a^{m^*+2-i} - \frac{a b^{m^*+1-i}}{1+\nu}\right),\hspace{.02in} i = 2, \dots, m^* + 1,
\end{eqnarray*}
and $y_{i,j}^* = 0$ when $j \neq m^*, m^* + 1$,
where $\nu \geq 0$ is the unique solution of the equation
$y_{0,m^*} + y_{0,m^*+1} = 1 - \lambda$, i.e.,
\begin{eqnarray}\label{eq:h}
h(\nu) = a^{m^*+1}+\frac{a^2 b^{m^*-1}\d^2}{(1+\nu)(\d+\nu)} = 1 - \lambda.
\label{nueq}
\end{eqnarray}
In particular, if $\lambda = 1 - (1 + \delta)^{- m^*}$, i.e.,
\begin{equation}
\label{eq:mstar}
m^* = - \frac{\log(1 - \lambda)}{\log(1 + \delta)},
\end{equation}
then $\nu = 0$, so
\begin{align*}
\begin{split}
y_{0,m^*}^* &=
\left(\frac{1}{1 + \delta}\right)^{m^*} = 1 - \lambda, \\
y_{i,m^*}^*
&=
\delta \left(\frac{1}{1 + \delta}\right)^{m^* + 1 - i} =
\left(1-(1-\lambda)^{1/m^*}\right)(1-\lambda)^{\frac{m^*-i}{m^*}}, \\
&\hspace*{1.3in} i = 1, \dots, m^*, 
\end{split}
\end{align*}
and $y_{i,m^*+1}^* = 0$ for all $i = 0, \dots, m^* + 1$.

\end{proposition}

Note that $h(u)$ is strictly decreasing in~$\nu$, with
$\lim_{\nu \downarrow 0} h(\nu) = (\frac{1}{1 + \delta})^{m^*} \leq 1 - \lambda$,
and $\lim_{\nu \to \infty} h(\nu) = (\frac{1}{1 + \delta})^{m^* + 1} <
1 - \lambda$, ensuring that $\nu \geq 0$ exists and is unique.

The result of Proposition~\ref{fp} is obtained by setting the
derivatives in~\eqref{eq:flasyn} equal to zero, observing that
$w_j(y^*) = 0$ for all $j \neq m^*, m^* + 1$, and then solving the
resulting equations.
The detailed proof arguments are presented in Appendix~\ref{deri}.

\begin{corollary}[No-queueing threshold for $\d$ in $\aujsqe$]
\label{cor:async1}
If the update frequency $\d \geq \lambda/(1-\lambda)$, then
$y_{i,j}^*=0$ for all $j \geq 2$, implying that queueing vanishes
at fluid level in stationarity.
\end{corollary}

Corollary~\ref{cor:async1} immediately follows from~\eqref{eq:mstarexp} and Proposition~\ref{fp}, in which $m^*=0$ in case $\delta \geq \lambda/(1-\lambda)$ so that only $y^*_{0,0}$, $y^*_{0,1}$ and $y^*_{1,1}$ are strictly positive. In case of equality we have the scenario described in the last part of Proposition~\ref{fp} where $y^*_{i,2}=0$ for all~$i$.
In case the many-server ($N \to \infty$) and stationary ($t \to \infty$)
limits can be interchanged (a rigorous proof of that would involve
establishing global asymptotic stability of the fluid limit, which
is beyond the scope of the present paper), Corollary~\ref{cor:async1}
implies that for $\d \geq \lambda/(1-\lambda)$, the mean stationary
waiting time under $\aujsqe$ vanishes as $N \to \infty$.

Proposition~\ref{fp} also yields an upper bound for the queue length
at fluid level as stated in the next corollary.

\begin{corollary}[Bounded queue length for $\aujsqe$]
\label{coro:asyn1}
The queue length at fluid level in stationarity has bounded support
on $\{0, \dots,$ $ m(\lambda, \delta) + 1\}$ for any $\lambda<1$ and $\delta>0$.
\end{corollary}

First of all, note that in order for queueing to vanish,
it is required that $m(\lambda, \delta) = 0$ or $m(\lambda, \delta) = 1$
and $\nu = 0$, i.e., $\delta \geq \lambda / (1 - \lambda)$,
which coincides with the threshold for $\delta$ in $\sujsqd$
as identified in Proposition~\ref{prop:sync1}.
Also, the upper bound $m(\lambda, \delta) + 1$ tends to infinity
as $\delta$ approaches zero, reflecting that for any fixed arrival
rate~$\lambda$, even arbitrarily low, the maximum queue length
grows without bound as the update frequency vanishes.

At the same time, for any positive $\delta > 0$,
$m(\lambda, \delta)$ is finite for any fixed $\lambda < 1$,
and only grows as $\log(1 / (1 - \lambda))$ as $\lambda \uparrow 1$
rather than $1 / (1 - \lambda)$ as in the absence of any queue feedback.
Thus, even an arbitrarily low update frequency ensures that the
queue length has bounded support and behaves far more benignly
in a high-load regime at fluid level.
This powerful property resembles an observation in work of
Tsitsiklis \& Xu \cite{TX12,TX13} in the context of a dynamic
scheduling problem where even a minuscule degree of resource
pooling yields a fundamentally different behavior on fluid scale.

\subsubsection{Number of jobs in the system}

The average queue length in the fixed point $\tilde{q}$ is
\begin{equation}\label{eq:jobs}
\begin{split}
&\tilde{q} = \sum_{i=1}^{m^*} i y_{i,m^*}^* + \sum_{i=1}^{m^*+1} i y_{i,m^*+1}^*\\
&=\frac{\delta +a^{m^*+1}+\delta  m^*-1}{\delta}+\frac{a^2 \delta   \left(-\delta	+b^{m^*}-1\right)}{b(1+\nu)	(\delta +v)}\\
&\overset{\eqref{eq:h}}{=}m^*+1+\frac{a^{m^*+1}-1}{\d}+\frac{1-\lambda-a^{m^*+1}}{\d}-\frac{a(1+\d+\nu)\d}{(1+\nu)(\d+\nu)}\\
&= m^*+1 -\lambda/\d-\frac{(1+\d+\nu) \d}{(1+\d)(1+\nu)(\d+\nu)} \geq m^* - \lambda/\d.
\end{split}
\end{equation}
In case of~\eqref{eq:mstar}, the average queue length is simply
$
\tilde{q} =m^*+1-\lambda/\d-1 $ $=
m(\lambda, \delta) - \lambda / \delta$,
reflecting that the average number of job arrivals equals
the average number of job completions over the course of an update
interval, starting with $m^* = m(\lambda, \delta)$ jobs.

Figure \ref{fig:qtilde} plots the average queue length in the fixed point given
by \eqref{eq:jobs} as function of the update frequency~$\delta$
for $\lambda = 0.8$.
We observe that the average queue length monotonically decreases
with the update frequency, as expected, and is indeed contained between
$m^* - \lambda / \delta$ and $m^* + 1 - \lambda/ \delta$.

It then follows from the definition of $m^* = m(\lambda, \delta)$
that $\tilde{q} \to \infty$ as $\delta \downarrow 0$ for any
$\lambda < 1$, which indicates that $\aujsqd$ may perform
arbitrarily badly in the ultra-low feedback regime,
confirming the observations in Subsection~\ref{benchmarks}.

\subsubsection{Bound on the queue length for $\aujsqd$}

As noted earlier, the fluid-limit trajectory for $\aujsqd$ involves
a measure-valued process and is difficult to describe.
However, in a similar spirit as for $\aujsqe$, the value of~$m^*$
can be characterized as the largest integer for which
\[
\left(\sum_{i = 1}^{m} i \frac{(1/\d)^i}{i!} \e^{-1/\d} +
m \sum_{i = m + 1}^{\infty} \frac{(1/\d)^i}{i!} \e^{-1/\d}\right) \leq
\lambda / \delta,
\]
expressing that the average number of job arrivals should be larger
than or equal to the average number of job completions over the
course of an update interval, starting with $m^*$~jobs.
While the above equation cannot easily be solved in closed form,
it is not difficult to show that the inequality is weaker
than~\eqref{eq:mstarexp}, i.e., the value of~$m^*$ is lower than
for $\aujsqe$, confirming the superiority of $\aujsqd$ observed
in the simulation results in Subsection~\ref{benchmarks}.
It is further worth observing the strong similarity of the above
inequality with Proposition \ref{prop:sync2} governing the queue length upper
bound for $\sujsqd$.

\begin{figure}
	\begin{center}
		\includegraphics[width=.5\linewidth]{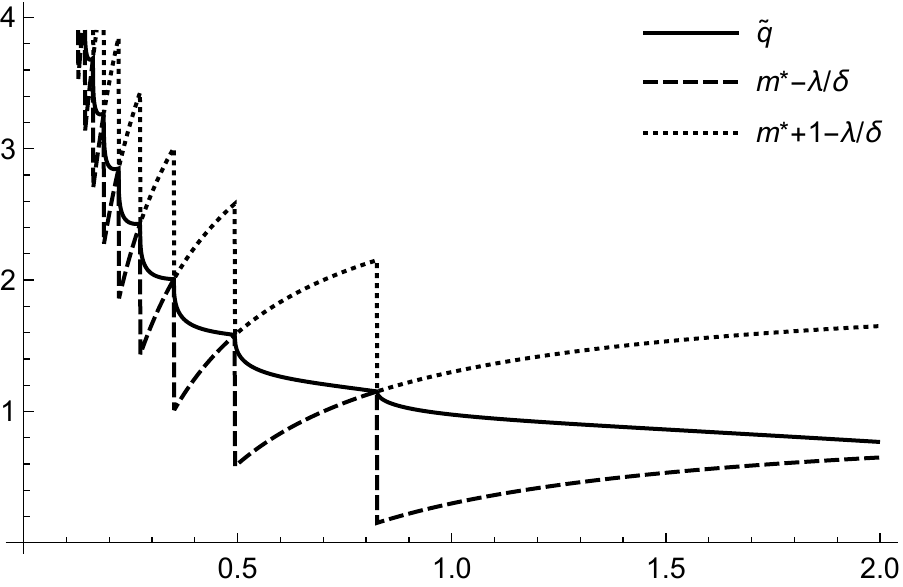}
	\end{center}
	\caption{Values of $\tilde{q}$ for $\lambda=0.7$ and different values of $\d$.}\label{fig:qtilde}
\end{figure}

\section{Conclusions}
\label{sec:conc}

We have introduced and analyzed a novel class of hyper-scalable
load balancing algorithms that only involve minimal communication
overhead and yet deliver excellent performance.
In the proposed schemes, the various servers provide occasional
queue status notifications so as to guide the dispatcher
in directing incoming jobs to relatively short queues.

We have demonstrated that the schemes markedly outperform JSQ($d$)
policies with a comparable overhead, and can drive the waiting
time to zero in the many-server limit with just one message per job.
The proposed schemes show their core strength and outperform sparsified
JIQ versions in the sparse feedback regime with less than one message per job,
which is particularly pertinent from a scalability viewpoint. 

In order to further explore the performance in the
many-server limit, we investigated fluid limits for synchronous
 as well as asynchronous exponential update intervals.
We used the fluid limits to obtain upper bounds for the stationary
queue length as function of the load and update frequency.
We also revealed a striking dichotomy in the ultra-low feedback
regime where the mean waiting time tends to a constant in the
synchronous case, but grows without bound in the asynchronous case.
Extensive simulation experiments are conducted to support the
analytical results, and indicate that the fluid-limit asymptotics
are remarkable accurate. 

In the present paper we have adopted common Markovian assumptions,
and in future work we aim to extend the results to non-exponential
and possibly heavy-tailed distributions.
We also intend to pursue schemes that may dynamically suppress
updates or selectively refrain from updates at pre-scheduled epochs
to convey implicit information, and reduce the communication
overhead yet further.

\section*{Acknowledgments}
This work is supported by the NWO Gravitation Networks grant 024.002.003, an NWO TOP-GO grant and an ERC Starting Grant.

\appendix

\section{Derivation sketch of fluid limit~(\ref{eq:fl})
for synchronous updates.}
\label{sketchsync}

We now provide an informal outline of the derivation of the fluid
limit for synchronous updates as stated in~\eqref{eq:fl}.
Let $A_{i,j}(t)$ and $S_{i,j}(t)$ denote unit-rate Poisson processes,
$j \geq i \geq 0$, all independent.

The system dynamics (in between successive update moments) may then be
represented as (see for instance~\cite{PTW07})
\[
\begin{split}
Y_{i,j}^N(t) =& ~ Y_{i,j}^N(t_0)  + S_{i+1,j}\left(\int_{t_0}^{t} Y_{i+1,j}^N(s) ds\right) \indi{i < j} \\
&- S_{i,j}\left(\int_{t_0}^{t} Y_{i,j}^N(s) ds\right) \indi{i > 0} \\
&+ A_{i-1,j-1}\left(\lambda N \int_{t_0}^{t} p_{i-1,j-1}(Y^N(s)) ds\right) \indi{i > 0} \\
&- A_{i,j}\left(\lambda N \int_{t_0}^{t} p_{i,j}(Y^N(s)) ds\right),
\end{split}
\]
with $p_{i,j}(Y) = \frac{Y_{i,j}}{\sum_{k=0}^{j} Y_{k,j}} \indi{j = m(Y)}$.


Dividing by~$N$ and rewriting in terms of the fluid-scaled variables
$y_{i,j}^N(t) = \frac{1}{N} Y_{i,j}^N(t)$, we obtain
\begin{equation}
\begin{split}
y_{i,j}^N(t) =& ~ y_{i,j}^N(t_0) \label{martingalesync} \\
&+ \frac{1}{N} S_{i+1,j}\left(N \int_{t_0}^{t} y_{i+1,j}^N(s) ds\right) \indi{i < j} \\
&- \frac{1}{N} S_{i,j}\left(N \int_{t_0}^{t} y_{i,j}^N(s) ds\right) \indi{i > 0} \\
&+ \frac{1}{N} A_{i-1,j-1}\left(\lambda N \int_{t_0}^{t} p_{i-1,j-1}(Y^N(s)) ds\right) \indi{i > 0} \\
&- \frac{1}{N} A_{i,j}\left(\lambda N \int_{t_0}^{t} p_{i,j}(Y^N(s)) ds\right).
\end{split}
\end{equation}

Now introduce
\[
\tilde{S}_{k,l}(u) := S_{k,l}(u) - u,~~~
\tilde{A}_{k,l}(u) := A_{k,l}(u) - u,
\]
and observe that $\tilde{S}_{k,l}(\cdot)$ and $\tilde{A}_{k,l}(\cdot)$
are martingales.
By standard arguments it can be shown that both
$\frac{1}{N} \tilde{S}_{k,l}(N \int_{t_0}^{t} y_{k,l}^N(s) ds)$
and $\frac{1}{N} \tilde{A}_{k,l}(\lambda N \int_{t_0}^{t} p_{k,l}(Y^N(s)) ds)$
converge to zero as $N \to \infty$. 

Exploiting the fact that the minimum queue estimate $m(Y^N(t))$ cannot
decrease in between successive update moments, it can also be
established that
\[
\int_{t_0}^{t} p_{k,l}(Y^N(s)) ds \to \int_0^t p_{k,l}(s) ds,
\]
as $N \to \infty$, with $p_{k,l}(y) = \frac{y_{k,l}}{w_l(y)} \indi{m(y) = l}$
as defined earlier.

Taking the limit for $N \to \infty$ in~\eqref{martingalesync},
we conclude that any (weak) limit $\{y_{i,j}(t)\}_{t \geq 0}$ of the
sequence $\left(\{y_{i,j}^N(t)\}_{t \geq 0}\right)_{N \geq 1}$
in between successive update moments must satisfy
\[
\begin{split}
y_{i,j}(t) =& ~ y_{i,j}(t_0) + \int_{t_0}^{t} y_{i+1,j}(s) ds \indi{i < j} -
\int_{t_0}^{t} y_{i,j}(s) ds \indi{i > 0} \\
&+ \lambda \int_{t_0}^{t} p_{i-1,j-1}(y(s)) ds \indi{i > 0} -
\lambda \int_{t_0}^{t} p_{i,j}(y(s)) ds,
\end{split}
\]
with $y_{i,j}(0) = y_{i,j}^\infty$.

Rewriting the latter integral equation in differential form yields~\eqref{eq:fl}.

\section{Derivation sketch of fluid limit~(\ref{eq:fl})
for asynchronous updates.}
\label{sketchasyn}

We now provide an informal outline of the derivation of the fluid
limit as stated in~\eqref{eq:flasyn}.
Let $A_{i,j}(t)$, $B_{i,j}(t)$ and $S_{i,j}(t)$ denote unit-rate Poisson
processes, $j \geq i \geq 0$, all independent.
The system dynamics may then be represented as~\cite{PTW07}
\[
\begin{split}
Y_{i,j}^N(t) =& ~ Y_{i,j}^N(0)  
+ S_{i+1,j}\left(\int_0^t Y_{i+1,j}^N(s) ds\right) \indi{i < j} \\
&- S_{i,j}\left(\int_0^t Y_{i,j}^N(s)ds\right) \indi{i > 0} \\
&+ A_{i-1,j-1}\left(\lambda N \int_0^t p_{i-1,j-1}(Y^N(s)) ds\right) \indi{i > 0}  \\
&- A_{i,j}\left(\lambda N \int_0^t p_{i,j}(Y^N(s)) ds\right) \\
&+  \sum_{k = j}^{\infty} B_{i,k}\left(\delta \int_0^t Y_{i,k}^N(s) ds\right) \indi{i = j} \\
&- B_{i,j}\left(\delta \int_0^t Y_{i,j}^N(s) ds\right),
\end{split}
\]
with $p_{i,j}(Y)$ as before.
Dividing by~$N$ and rewriting in terms of the fluid-scaled variables
$y_{i,j}^N(t) = \frac{1}{N} Y_{i,j}^N(t)$, we obtain
\begin{equation}
\begin{split}
y_{i,j}^N(t) &= ~ y_{i,j}^N(0) \label{martingaleasyn} 
+ \frac{1}{N} S_{i+1,j}\left(N \int_0^t y_{i+1,j}^N(s) ds\right) \indi{i < j} \\
&- \frac{1}{N} S_{i,j}\left(N \int_0^t y_{i,j}^N(s)ds\right) \indi{i > 0} \\
&+ \frac{1}{N} A_{i-1,j-1}\left(\lambda N \int_0^t p_{i-1,j-1}^N(Y^N(s))\right) \indi{i > 0} \\
&- \frac{1}{N} A_{i,j}\left(\lambda N \int_0^t p_{i,j}^N(Y^N(s)) ds\right) \\
&+  \frac{1}{N} \sum_{k = j}^{\infty} B_{i,k}\left(\delta N \int_0^t y_{i,k}^N(s) ds\right) \indi{i = j} \\
&- \frac{1}{N} B_{i,j}\left(\delta N \int_0^t y_{i,j}^N(s) ds\right).
\end{split}
\end{equation}
Now introduce
\[
\tilde{S}_{k,l}(u) := S_{k,l}(u) - u,~~
\tilde{A}_{k,l}(u) := A_{k,l}(u) - u,~~
\tilde{B}_{k,l}(u) := B_{k,l}(u) - u,
\]
and observe that $\tilde{S}_{k,l}(\cdot)$, $\tilde{A}_{k,l}(\cdot)$
and $\tilde{B}_{k,l}(\cdot)$ are martingales.
By standard arguments it can be shown that 
\[
\begin{split}
&\frac{1}{N} \tilde{S}_{k,l}\left(N \int_0^t y_{k,l}^N(s) ds\right), \\
&\frac{1}{N} \tilde{A}_{k,l}\left(\lambda N \int_0^t p_{k,l}(Y^N(s)) ds\right), \\
&\frac{1}{N} \tilde{B}_{k,l}\left(\delta N \int_0^t y_{k,l}^N(s) ds\right)
\end{split}
\]
each converge to zero as $N \to \infty$.

Adopting time-scale separation arguments as developed by Hunt
\& Kurtz~\cite{HK94}, it can be established that
\[
\lambda \int_0^t p_{i,j}(Y^N(s)) ds \to \int_0^t \alpha_{i,j}(s) ds
\]
as $N \to \infty$, where the coefficients $\alpha_{i,j}(\cdot)$ satisfy
\[
\alpha_{i,j}(t) = \left\{\begin{array}{ll}
0 & i < j < n(t), \\
\lambda \pi_j(t) & i = j < n(t), \\
\lambda \frac{y_{i,j}(t)}{w_j(y(t))} \pi_{n(t)}(t) & i \leq j = n(t), \\
0 & i \leq j > n(t).
\end{array} \right.
\]
The coefficients $\pi_j(t)$ may be interpreted as the fraction of time
that the pre-limit minimum queue estimate equals~$j \leq n(t)$ when the
minimum queue estimate at fluid level is~$n(t)$, and satisfy the
relationship
\[
\lambda \pi_j(t) = \lambda \pi_{j-1}(t) + \delta v_j(y(t))
\]
for all $j = 1, \dots, n(t) - 1$, along with the normalization condition
$\sum_{j = 0}^{n(t)} \pi_j(t) = 1$.

Thus, we obtain
\[
\lambda \pi_j(t) = \delta \sum_{k = 0}^{j} v_k(y(t)),
\]
for all $j = 0, \dots, n(t) - 1$, and
\[
\begin{split}
\lambda \pi_{n(t)}(t) &= \lambda (1 - \sum_{j = 0}^{n(t) - 1} \pi_j(t)) =
\lambda - \delta \sum_{j = 0}^{n(t) - 1} \sum_{k = 0}^{j} v_k(y(t))\\ &=
\lambda - \delta \sum_{i = 0}^{n(t) - 1} (n(t) - i) v_i(y(t)) = \zeta(y(t)).
\end{split}
\]

We deduce that
\[
\alpha_{i,j}(t) =
\indi{i = j} \delta \sum_{k = 0}^{j} v_k(y(t)) \indi{j < n(t)} +
q_{i,j}(y(t) \zeta(y(t)),
\]
with $q_{i,j}(y) = \frac{y{i,j}}{w_j(y)} \indi{n(y) = j}$ as before,
yielding
\[
\begin{split}
&\int_0^t \alpha_{i-1,j-1}(s) ds - \int_0^t \alpha_{i,j}(s) ds \\
=&
\int_0^t q_{i-1,j-1}(y(s)) \zeta(y(s)) ds - \int_0^t q_{i,j}(y(s)) \zeta(y(s)) ds \\&+ 
\indi{i = j} \delta \sum_{k = 0}^{j - 1} v_k(y(s)) \indi{j = n(s)} ds \\&-
\indi{i = j} \delta \int_0^t v_j(y(s)) \indi{j < n(s)} ds.
\end{split}
\]

Thus we obtain
\begin{equation}\label{hk1}
\begin{split}
&\lambda \int_0^t p_{i-1,j-1}(Y^N(s)) ds \indi{i > 0} -
\lambda \int_0^t p_{i,j}(Y^N(s)) ds \to  \\
&~~\int_0^t q_{i-1,j-1}(y(s)) \zeta(y(s)) ds \indi{i > 0} -
\int_0^t q_{i,j}(y(s)) \zeta(y(s)) ds \\
&~~+ \indi{i = j} \delta \sum_{k = 0}^{j - 1} v_k(y(s)) \indi{j = n(y(s))} ds \\
&~~-\indi{i = j} \delta \int_0^t v_j(y(s)) \indi{j < n(y(s))} ds,
\end{split}
\end{equation}
as $N \to \infty$,
with $\zeta(y) = \lambda - \delta \sum_{l = 0}^{n(y)} (n(y) - l) v_l(y)$
and $q_{i,j}(y) = \frac{y_{i,j}}{w_j(y)} \indi{n(y) = j}$
as defined earlier.

Taking the limit for $N \to \infty$ in~\eqref{martingaleasyn},
and noting that
\[
\begin{split}
&\indi{i = j} \delta \int_0^t v_i(y(s)) ds -
\indi{i = j} \delta \int_0^t v_j(y(s)) \indi{j < n(y(s))} ds \\
&=
\delta \int_0^t v_i(y(s)) \indi{i = j \geq n(y(s))},
\end{split}
\]
we conclude that any (weak) limit $\{y_{i,j}(t)\}_{t \geq 0}$ of the
sequence $(\{y_{i,j}^N(t)\}_{t \geq 0})_{N \geq 1}$ must satisfy
\[
\begin{split}
& y_{i,j}(t) =  y_{i,j}(0) + \int_0^t y_{i+1,j}(s) ds \indi{i < j} -
\int_0^t y_{i,j}(s) ds \indi{i > 0} \\
&~+ \lambda \int_0^t q_{i-1,j-1}(y(s)) \zeta_{j-1}(y(s)) ds \indi{i > 0}\\
&~ -
\lambda \int_0^t q_{i,j}(y(s)) \zeta_j(y(s)) ds + \delta \sum_{k = 0}^{i - 1} \int_0^t v_k(y(s)) \indi{i = j = n(y)} ds \\
&~+
\delta \int_0^t v_i(y(s)) \indi{i = j \geq n(y(s))} - \delta \int_0^t y_{i,j}(s) ds,
\end{split}
\]
with $y_{i,j}(0) = y_{i,j}^\infty$.
Rewriting the latter integral equation in differential form yields~\eqref{eq:flasyn}.

\section{Proofs of Section~\ref{subsec:fpsync}}
\label{app:sec42}

\subsection{Proof of Lemma \ref{lem:fund}}
\label{app:lemfund}

Let $\tilde{y}_{i,j}(t)$, $i = 0, 1, \dots, j$, $j \geq 0$, be the solution
to the fluid-limit equation~\eqref{eq:fl} with $\lambda = 0$, i.e.,
\[
\frac{\tilde{y}_{i,j}(t)}{dt} =
\tilde{y}_{i+1,j}(t) \mathds{1}\{i < j\} - \tilde{y}_{i,j}(t) \mathds{1}\{i>0\},
\]
with initial conditions $\tilde{y}_{i,i}(0) = v_i(0)$
and $\tilde{y}_{i,j}(0) = 0$ for all $j \geq i + 1$, $i \geq 0$.
The solution $\tilde{y}_{i,j}(t)$ may be interpreted as the fluid limit
in the absence of any arrivals, and it is easily verified that
\[
\tilde{y}_{i,j}(t) = v_j(0) \frac{t^{j - i}}{(j - i)!} \ee^{- t},
\]
for all $i = 1, 2, \dots, j$, and
\[
\tilde{y}_{0,j}(t) = v_j(0) \sum_{k = j}^{\infty} \frac{t^k}{k!} \ee^{- t},
\]
$j \geq 0$.
Further introduce $\tilde{v}_i(t) = \sum_{j = i}^{\infty} \tilde{y}_{i,j}(t)$,
$\tilde{z}_k(t) = \sum_{i = k}^{\infty} \tilde{v}_i(t)$,
$\tilde{Q}^{> K}(t) = \sum_{k = K + 1}^{\infty} \tilde{z}_k(t)$,
and note from~\eqref{derivativequeuemass1} that
\begin{equation}
\frac{d \tilde{Q}^{> K}(t)}{dt} = - \tilde{z}_{K + 1}(t).
\label{derivativequeuemass3}
\end{equation}

We will first establish that $\tilde{z}_k(t) \leq z_k(t)$ for all
$k \geq 0$, $t \in [0, T)$, reflecting that the fraction of servers
with queue length~$k$ or larger on fluid scale is no less than what
it would be in the absence of any arrivals.
Suppose that were not the case, and let $t_0 \in [0, T)$ be the first
time when that inequality is about to be violated for some $k_0 > 0$.
Then we must have $z_{k_0}(t_0) = \tilde{z}_{k_0}(t_0)$, implying
$v_{k_0}(t_0) = z_{k_0}(t_0) - z_{k_0 + 1}(t_0) \leq
\tilde{z}_{k_0}(t_0) - \tilde{z}_{k_0 + 1}(t_0) = \tilde{v}_{k_0}(t_0)$,
since $\tilde{z}_{k_0 + 1}(t_0) \leq z_{k_0 + 1}(t_0)$.
Now observe that
\[
\frac{d z_{k_0}(t)}{dt} \mid_{t = t_0} =
- v_{k_0}(t_0) + \lambda \sum_{j = k_0 - 1}^{\infty} p_{k_0 - 1, j}(t_0) \geq
- v_{k_0}(t_0),
\]
while
\[
\frac{d \tilde{z}_{k_0}(t)}{dt} \mid_{t = t_0} =
- \tilde{v}_{k_0}(t_0) \leq - v_{k_0}(t_0) \leq \frac{d z_{k_0}(t)}{dt} .
\]
Hence $z_{k_0}(t)$ cannot fall below $\tilde{z}_{k_0}(t)$ at (just after) $t_0$, contradicting the
initial supposition in which $t_0$ would be the first
time that the inequality is about to be violated.

Invoking~\eqref{derivativequeuemass3}, we obtain
\[
\begin{split}
\int_{s = 0}^{t} z_{K + 1}(s) ds &\geq
\int_{s = 0}^{t} \tilde{z}_{K + 1}(s) ds = Q^{> K}(0) - \tilde{Q}^{> K}(t) \\
&=
Q^{> K}(0) - \sum_{k = K + 1}^{\infty} (k - K) \tilde{v}_k(t) \\
&=
Q^{> K}(0) - \sum_{k = K + 1}^{\infty} (k - K) \sum_{l =
k}^{\infty} \tilde{y}_{k,l}(t) \\
&=
Q^{> K}(0) - \sum_{k = K + 1}^{\infty} (k - K) \sum_{l = k}^{\infty} 
v_l(0) \frac{t^{l - k}}{(l - k)!} \ee^{- t} \\
&=
Q^{> K}(0) - \sum_{l = K + 1}^{\infty} v_l(0) \sum_{k = K + 1}^{l} (k - K)
\frac{t^{l - k}}{(l - k)!} \ee^{- t} \\
&=
Q^{> K}(0) - \sum_{l = K + 1}^{\infty} v_l(0) \sum_{m = 0}^{l - K - 1}
(l - K - m) \frac{t^m}{m!} \ee^{- t} \\
&=
Q^{> K}(0) - \sum_{l = K + 1}^{\infty} v_l(0) A(l - K, t) \\
&=
Q^{> K}(0) - \sum_{l = 1}^{\infty} v_{K + l}(0) A(l, t).
\end{split}
\]

Now observe that
\[
\begin{split}
\ee^t A(L, t) &= \sum_{l = 0}^{L} (L - l) \frac{t^l}{l!} =
\sum_{l = 0}^{L} \frac{L - l}{L + 1 - l} (L + 1 - l) \frac{t^l}{l!} \\
&\leq
\sum_{l = 0}^{L} \frac{L}{L + 1} (L + 1 - l) \frac{t^l}{l!} =
\frac{L}{L + 1} \sum_{l = 0}^{L} (L + 1 - l) \frac{t^l}{l!} \\
&=
\frac{L}{L + 1} \sum_{l = 0}^{L + 1} (L + 1 - l) \frac{t^l}{l!} =
\frac{L}{L + 1} \ee^t A(L + 1, t),
\end{split}
\]
or equivalently,
\[
\frac{A(L, t)}{L} \leq \frac{A(L + 1, t)}{L + 1},
\]
which may be interpreted from the fact that the expected fraction
of jobs that remain after a period of length~$t$ is smaller
with an initial queue of size~$L$ than $L + 1$.
Thus, $A(l, t) \leq \frac{l}{L} A(L, t)$ for all $l \leq L$.
Also,
\[
\begin{split}
A(L, t) &= \sum_{l = 0}^{L} (L - l) \frac{t^l}{l!} \ee^{- t} =
\sum_{l = 0}^{L - 1} (L - 1 - l) \frac{t^l}{l!} \ee^{- t} +
\sum_{l = 0}^{L - 1} \frac{t^l}{l!} \ee^{- t} \\
&\leq
A(L - 1, t) + 1,
\end{split}
\]
so that $A(l, t) \leq l - L + A(L, t)$ for all $l \geq L + 1$.
We obtain
\[
\begin{split}
\sum_{l = 1}^{\infty}& v_{K + l}(0) A(l, t) =
\sum_{l = 1}^{L} v_{K + l}(0) A(l, t) +
\sum_{l = L + 1}^{\infty} v_{K + l}(0) A(l, t) \\
\leq&
\sum_{l = 1}^{L} v_{K + l}(0) \frac{l}{L} A(L, t) +
\sum_{l = L + 1}^{\infty} v_{K + l}(0) [l - L + A(L, t)] \\
=&
\frac{1}{L} \left[\sum_{l = 1}^{L} l v_{K + l}(0) +
\sum_{l = L + 1}^{\infty} v_{K + l}(0)\right] A(L, t) +
\sum_{l = L + 1}^{\infty} (l - L) v_{K + l}(0) \\
=&
Q^{> K + L}(0) +
\frac{1}{L} \left[Q^{\leq K + L}(0) - Q^{\leq K}(0)\right] A(L, t),
\end{split}
\]
yielding
\[
\begin{split}
&Q^{> K}(0) - \sum_{l = 1}^{\infty} v_{K + l}(0) A(l, t) \\
&\geq
Q^{> K}(0) - Q^{> K + L}(0) -
\frac{1}{L} \left[Q^{\leq K + L}(0) - Q^{\leq K}(0)\right] A(L, t) \\
&=
Q^{\leq K + L}(0) - Q^{\leq K}(0) -
\frac{1}{L} \left[Q^{\leq K + L}(0) - Q^{\leq K}(0)\right] A(L, t) \\
&=
\left[Q^{\leq K + L}(0) - Q^{\leq K}(0)\right] \left[1 - \frac{A(L, t)}{L}\right].
\end{split}
\]

\subsection{Proof of Lemma \ref{auxiliary}}\label{app:lemms}

We have $m(t) = 0$, $\frac{d}{dt} w_0(t) $ $= - \frac{d}{dt} w_1(t) $ $= -
\lambda$ for $t \in \frac{1}{\lambda} [0, v_0(0)]$,
$m(t) = 1$, $\frac{d}{dt} w_1(t) $ $= - \frac{d}{dt} w_2(t) $ $ = - \lambda$
for $t \in \frac{1}{\lambda} [v_0(0), 2 v_0(0) + v_1(0)] $,
$m(t) = 2$, $\frac{d}{dt} w_2(t) $ $= - \frac{d}{dt} w_3(t) $ $ = - \lambda$
for $t \in \frac{1}{\lambda}
[2 v_0(0) + v_1(0), 3 v_0(0) + 2v_1(0) + v_2(0)]$, $\dots$,
$m(t) = j$, $\frac{d}{dt} w_j(t) $ $= - \frac{d}{dt} w_{j+1}(t) $ $= - \lambda$
for $t \in \frac{1}{\lambda} \big[\sum_{i = 0}^{j - 1} (j - i) v_i(0), $ $
\sum_{i = 0}^{j} (j + 1 - i) v_i(0)\big]$, 
assuming $\frac{1}{\lambda} \sum_{i = 0}^{j} (j + 1 - i) v_i(0) \leq t$.
In particular $m(s) \leq K - 1$ for all $s \in [0, t)$
if $\lambda t \leq \sum_{i = 0}^{K - 1} (K - i) v_i(0)$.
The latter inequality holds since
\[
\begin{split}
\sum_{i = 0}^{K - 1} (K - i) v_i(0) &= 
K \sum_{i = 0}^{K - 1} v_i(0) - \sum_{i = 0}^{K - 1} i v_i(0) \\
&=
K \left(1 - \sum_{i = K}^{\infty} v_i(0)\right) -
\sum_{i = 0}^{K - 1} i v_i(0) \\
&= K - Q^{\leq K}(0) \geq K - Q(0) \geq \lambda t.
\end{split}
\]

\subsection{Proof of Corollary~\ref{corollary1}}
\label{c1}

Taking $K = 0$ in Lemma~\ref{basic2}, we obtain
\[
\begin{split}
Q(T) &\leq Q(0) + \lambda T - \frac{Q^{\leq L}(0)}{L} B(L, T) \\
&\leq Q(0) + \lambda T - \frac{L - 1 - \lambda T - Q(0) + Q^{\leq L}(0)}{L} B(L, T) \\
&= Q(0) - \Delta_L + \frac{Q^{> L}(0)}{L} B(L, T)  \leq Q(0) - \Delta + \frac{Q^{> L}(0)}{L} B(L, T)
\end{split}
\]
with
\begin{equation}\label{eq:deltaL}
\Delta_L = \left(1 - \frac{\lambda T + 1}{L}\right) B(L, T) - \lambda T
\end{equation}
increasing in $L$, $\Delta = \Delta_{L^*}$ and $L^*=s(\lambda,T)$ the smallest value that is large enough, note that $\Delta>0$ because of (\ref{eq:cond}).


\subsection{Proof of Corollary~\ref{corollary2}}
\label{c2}

Taking $K = 0$ in Lemma~\ref{basic2} and noting that
$1 - \frac{1}{L} B(L, T) \geq 0$ since $B(L,T) \leq L$, we obtain
\[
\begin{split}
Q(T) &\leq Q(0) + \lambda T - \frac{Q^{\leq L}(0)}{L} B(L, T) \\
&= Q(0) \left(1 - \frac{1}{L} B(L, T)\right) +
\frac{Q(0) - Q^{\leq L}(0)}{L} B(L, T) + \lambda T \\
&\leq (L - 1 - \lambda T) \left(1 - \frac{1}{L} B(L, T)\right) +
\frac{Q(0) - Q^{\leq L}(0)}{L} B(L, T) + \lambda T \\
&= L - 1 - \lambda T - \Delta_L + \frac{Q^{> L}(0)}{L} B(L, T) \\
&\leq L - 1 - \lambda T - \Delta + \frac{Q^{> L}(0)}{L} B(L, T),
\end{split}
\]
with $\Delta_L$ and $\Delta$ as defined in Appendix \ref{c1} and $\Delta > 0$ because of (\ref{eq:cond}).

\subsection{Proof of Lemma~\ref{lem:bounds}}
\label{app:lembounds}

Just like in the proof of Lemma~\ref{lem:fund},
let $\tilde{y}_{i,j}(t)$, $i = 0, 1, \dots, j$, $j \geq 0$, be the solution
to the fluid-limit equation~\eqref{eq:fl} with $\lambda = 0$, i.e.,
\[
\frac{\tilde{y}_{i,j}(t)}{dt} =
\tilde{y}_{i+1,j}(t) \mathds{1}\{i < j\} - \tilde{y}_{i,j}(t) \mathds{1}\{i>0\},
\]
but now with initial conditions such that $\tilde{z}_i(0) \leq z_i(0)$ for all
$i \geq 1$, with $\tilde{z}_i(t) = \sum_{k = i}^{\infty} \tilde{v}_k(t)$
and $\tilde{v}_k(0) = \tilde{y}_{k,k}(0)$,
and $\tilde{y}_{i,j}(0) = 0$ for all $j \geq i + 1$, $i \geq 0$.
As before, $\tilde{y}_{i,j}(t)$ may be interpreted as the fluid limit
in the absence of any arrivals, and it is easily verified that
\[
\tilde{y}_{i,j}(t) = \tilde{v}_j(0) \frac{t^{j - i}}{(j - i)!} \ee^{- t},
\]
for all $i = 1, 2, \dots, j$, and
\[
\tilde{y}_{0,j}(t) = \tilde{v}_j(0) \sum_{k = j}^{\infty} \frac{t^k}{k!} \ee^{- t},
\]
$j \geq 0$.
Further let $y_{0,1}^0(t)$, $y_{1,1}^0(t)$ be solutions to the system
of differential equations
\[
\begin{split}
\frac{d y_{0,1}^0(t)}{dt} &= y_{1,1}^0(t) \\
\frac{d y_{1,1}^0(t)}{dt} &= \lambda \mathds{1}\{t \leq t_0\} - y_{1,1}^0(t),
\end{split}
\]
with $t_0 = \min\{\tilde{v}_0(0) / \lambda, T\}$ and initial
conditions $y_{0,1}^0(0) = y_{1,1}^0(0) = 0$.

It is easily verified that
\[
\begin{split}
y_{0,1}^0(t) &=
\left\{\begin{array}{ll} \lambda [t - 1 + \ee^{- t}] & t \in [0, t_0], \\
\lambda [t_0 - \ee^{- (t - t_0)} + \ee^{- t}] & t \in [t_0, T],
\end{array}\right. \\
y_{1,1}^0(t) &=
\left\{\begin{array}{ll} \lambda [1 - \ee^{- t}] & t \in [0, t_0], \\
\lambda [\ee^{- (t - t_0)} - \ee^{- t}] & t \in [t_0, T].
\end{array}\right.
\end{split}
\]

The variable $y_{0,1}^0(t)$ may be interpreted as the fraction
of servers with queue length~$0$ at time~$0$, queue length~$0$
at time~$t$ and queue estimate~$1$ at time~$t$, i.e., which have
been assigned an arriving job and completed that job by time~$t$.
Likewise, $y_{1,1}^0(t)$ may be interpreted as the fraction
of servers with queue length~$0$ at time~$0$, queue length~$1$
at time~$t$ and queue estimate~$1$ at time~$t$, i.e., which have
been assigned an arriving job that remains to completed by time~$t$.

In a similar fashion as in the proof of Lemma~\ref{lem:fund},
it can be established that $\tilde{z}_1(t) + y_{1,1}^0(t) \leq z_1(t)$
and $\tilde{z}_i(t) \leq z_i(t)$ for all $i \geq 2$ and $t \in [0, T]$.

To prove statement~(i), consider
$\tilde{z}_1(0) = \min\{z_1(0), 1 - \lambda T\} \geq
\min\{1 - v_0(0), 1 - \lambda T\} \geq \min\{\lambda, 1 - \lambda T\} = \lambda$,
and $\tilde{z}_k(0) = 0$ for all $k \geq 2$.
Noting that $\tilde{v}_0(0) = 1 - \tilde{z}_1(0) \geq \lambda T$ yields
$t_0 = T$, and thus $y_{1,1}^0(T) = \lambda (1 - \ee^{- T})$.
Also, $\tilde{z}_1(T) = \tilde{y}_{1,1}(T) = \tilde{v}_1(0) \ee^{- T} =
\tilde{z}_1(0) \ee^{- T} \geq \lambda \ee^{- T}$.
We obtain that
\[
z_1(T) \geq \tilde{z}_1(T) + y_{1,1}^0(T) \geq \lambda,
\]
yielding $v_0(t) = 1 - z_1(T) \leq 1 - \lambda$.

To establish assertion~(ii), consider $\tilde{z}_k(0) = z_k(0)$
for all $k \geq 1$.
Then just like in the proof of Lemma~\ref{basic2},
noting that $A(l, T) \leq A(2, T) + l - 2$ for all $l \geq 2$,
\[
\begin{split}
\int_{t = 0}^{T}& \tilde{z}_1(t) dt =
\tilde{Q}(0) - \sum_{l = 1}^{\infty} v_l(0) A(l, T) \\
&=
\sum_{l = 1}^{\infty} l v_l(0) - \sum_{l = 1}^{\infty} v_l(0) A(l, T) =
\sum_{l = 1}^{\infty} v_l(0) [l - A(l, T)] \\
&=
\sum_{l = 1}^{\infty} v_l(0) [1 - A(1, T)] +
\sum_{l = 2}^{\infty} v_l(0) [l - 1 + A(1, T) - A(l, T)] \\
&\geq
z_1(0) [1 - A(1, T)] +
\sum_{l = 2}^{\infty} v_l(0) [1 + A(1, T) - A(2, T)] \\
&=
z_1(0) [1 - \ee^{- T}] + z_2(0) [1 - \ee^{- T} - T \ee^{- T}].
\end{split}
\]

Also, $t_0 = T$, and thus
\[
\int_{t = 0}^{T} y_{1,1}^0(t) dt = y_{0,1}^0(T) =
\lambda [T - 1 + \ee^{- T}].
\]

We obtain that
\[
\begin{split}
\int_{t = 0}^{T}& [1 - v_0(t)] dt = \int_{t = 0}^{T} z_1(t) dt \geq
\int_{t = 0}^{T} [y_{1,1}^0(t) + \tilde{z}_1(t)] dt \\
&\geq
\lambda [T - 1 + \ee^{- T}] +
z_1(0) [1 - \ee^{- T}] + z_2(0) [1 - \ee^{- T} - T \ee^{- T}] \\
&=
\lambda T + [z_1(0) - \lambda] [1 - \ee^{- T}] +
z_2(0) [1 - \ee^{- T} - T \ee^{- T}].
\end{split}
\]

To prove statement (iii), consider as before $\tilde{z}_k(0) = z_k(0)$
for all $k \geq 1$.
Further observe that
\[
A(l, T) \geq A(l - 1, T) + A(1, T),
\]
and
\[
A(k, T) - k = - B(k, T) \geq - T.
\]
Then just like in the proof of Lemma~\ref{basic2},
noting that $A(l, T) \leq A(2, T) + l - 2$ for all $l \geq 2$,
\[
\begin{split}
\tilde{Q}(T) =& \sum_{l = 1}^{\infty} v_l(0) A(l, T) =
\sum_{l = 1}^{\infty} v_l(0) l +
\sum_{l = 1}^{\infty} v_l(0) [A(l, T) - l] \\
=&
Q(0) + \sum_{l = 1}^{\infty} v_l(0) [A(1, T) - 1] \\
&+
\sum_{l = 2}^{\infty} v_l(0) [A(l, T) - l - A(1, T) + 1]\\
 \geq &
Q(0) + z_1(0) [A(1, T) - 1] +
\sum_{l = 2}^{\infty} v_l(0) [A(l - 1, T) - l + 1]\\
 \geq &
z_1(0) [1 - \ee^{- T}] - z_2(0) T.
\end{split}
\]

Also, $t_0 = T$, and thus
\[
y_{0,1}^0(T) = \lambda [T - 1 + \ee^{- T}].
\]

We obtain
\[
\begin{split}
Q(T) &\geq y_{0,1}^0(T) + \tilde{Q}(T) \\
&\geq
\lambda [T - 1 + \ee^{- T}] + z_1(0) [1 - \ee^{- T}] - z_2(0) T \\
&=
\lambda T + [z_1(0) - \lambda] [1 - \ee^{- T}] - z_2(0) T.
\end{split}
\]

To establish assertion~(iv), consider
$\tilde{z}_k(0) = \min\{z_k(0), 1 - \lambda T\}$ for all $k \geq 1$.

Then, just like in the proof of statement~(ii),
\[
\begin{split}
\int_{t = 0}^{T}& \tilde{z}_1(t) dt =
\tilde{Q}(0) - \sum_{l = 1}^{\infty} \tilde{v}_l(0) A(l, T) \\
&=
\sum_{l = 1}^{\infty} l \tilde{v}_l(0) -
\sum_{l = 1}^{\infty} \tilde{v}_l(0) A(l, T) =
\sum_{l = 1}^{\infty} \tilde{v}_l(0) [l - A(l, T)] \\
&=
\sum_{l = 1}^{\infty} \tilde{v}_l(0) [1 - A(1, T)] +
\sum_{l = 2}^{\infty} \tilde{v}_l(0) [l - 1 + A(1, T) - A(l, T)] \\
&\geq
\tilde{z}_1(0) [1 - A(1, T)] +
\sum_{l = 2}^{\infty} \tilde{v}_l(0) [1 + A(1, T) - A(2, T)] \\
&=
\tilde{z}_1(0) [1 - \ee^{- T}] + \tilde{z}_2(0) [1 - \ee^{- T} - T \ee^{- T}].
\end{split}
\]

Also, noting that $\tilde{v}_0(0) = 1 - \tilde{z}_1(0) \geq \lambda T$
yields $t_0 = T$, and thus $y_{1,1}^0(T) = \lambda (1 - \ee^{- T})$.

We obtain that
\[
\begin{split}
\int_{t = 0}^{T}& [1 - v_0(t)] dt = \int_{t = 0}^{T} z_1(t) dt \geq
\int_{t = 0}^{T} [y_{1,1}^0(t) + \tilde{z}_1(t)] dt \\
&\geq
\lambda [T - 1 + \ee^{- T}] +
\tilde{z}_1(0) [1 - \ee^{- T}] + \tilde{z}_2(0) [1 - \ee^{- T} - T \ee^{- T}] \\
&=
\lambda T + [\tilde{z}_1(0) - \lambda] [1 - \ee^{- T}] +
\tilde{z}_2(0) [1 - \ee^{- T} - T \ee^{- T}].
\end{split}
\]

Statement~(v) follows from statements~(ii) and~(iv).

\section{Derivation of fixed point}
\label{deri}

For convenience, denote by
$m^* = \min(j | w_j^* > 0)$ the minimum queue estimate
associated with the fixed point.
Further denote $n^* = m^*$ if $u_{m^* - 1}^* \leq \lambda$,
or $n^* = \min\{n: u_n^* > \lambda\}$ otherwise.

Setting the derivatives in \eqref{eq:flasyn} equal to zero
and denoting $\nu = \zeta / w_{n^*}$, we deduce
\begin{align}\label{eq:flasyncsettozero}
\begin{split}
0&=
y_{i+1,j}^* \mathds{1}\{i+1 \leq j\} - y_{i,j}^* \mathds{1}\{i>0\}\\
&+
\nu y_{i-1,j-1}^* \mathds{1}\{n^*=j-1\} -
\nu y_{i,j}^* \mathds{1}\{n^*=j\} \\
&+
\delta \sum_{k = 0}^{i} v_k^* \mathds{1}\{i = j = n^*\} +
\delta v_i^* \mathds{1}\{i = j \geq n^* + 1\} -
\delta y_{i,j}^*
\end{split}
\end{align}
for all $i = 0, 1, \dots, j \geq n^*$. 
Similarly, we have for $j_0\geq n^*+2$,
\begin{equation}\label{eq:dwfixedpoint}
0 = \frac{d}{dt} \sum_{j = j_0}^{\infty} w_j(t) =
\delta \sum_{j = j_0}^{\infty} [v_j^* - w_j^*] =
- \delta \sum_{i = 0}^{j_0 - 1} \sum_{j = j_0}^{\infty} y_{i,j}^*
\end{equation}
which yields $y_{i,j}^*=0$ for all $j\geq n^*+2$ and $i<j$.
Additionally, applying~(\ref{eq:flasyncsettozero}) with $i=k+1$
and $j=k+2$, gives 
\[
	0 =	y_{k+2,k+2}^* - y_{k+1,k+2}^* - \delta y_{k+1,k+2}^* = y_{k+2,k+2}^*
\]
for all $k\geq n$. In conclusion, it is readily seen that $w_j^*=0$ for all $j\geq n^* + 2$. This implies $m^* = n^*$, and yields
\begin{align*}
&y_{i+1,m^*}^* \mathds{1}\{i \neq m^*\} -
(\mathds{1}\{i \neq 0\} + \nu + \delta) y_{i,m^*}^* \\
&+
 \delta \sum_{k = 0}^{m^*} v_k^* \mathds{1}\{i = m^*\}=0,
\end{align*}
for all $i = 0, 1, \dots, m^*$, and
\begin{align*}
&y_{i+1,m^*+1}^*\mathds{1}\{i\neq m^*+1\} - (\mathds{1}\{i \neq 0\} + \delta) y_{i,m^*+1}^*\\ &+
\nu y_{i-1,m^*} \mathds{1}\{i \neq 0\} +
\delta y_{m^*+1,m^*+1}^* \mathds{1}\{i = m^* + 1\} = 0,
\end{align*}
for all $i = 0, 1, \dots, m^* + 1$.


We obtain (with $\mu \equiv 1$)
\begin{align*}
(\nu + \delta) y_{0,m^*}^* &= \mu y_{1,m^*}^* \\
(\mu + \nu + \delta) y_{i,m^*}^* &= \mu y_{i+1,m^*}^*,
\hspace*{.2in} i = 1, \dots, m^* - 1 \\
(\mu + \nu) y_{m^*,m^*}^* &= \delta \left[\sum_{i = 0}^{m^* - 1} y_{i,m^*}^* +
\sum_{i = 0}^{m^*} y_{i,m^*+1}^*\right],
\end{align*}
or equivalently,
\begin{align}
(\mu + \nu + \delta) y_{m^*,m^*}^* &=
\delta \sum_{i = 0}^{m^*} [y_{i,m^*}^* + y_{i,m^*+1}^*]\nonumber\\ 
&= \delta [1 - y_{m^*+1,m^*+1}^*] ,\label{eq:25071}
\end{align}
and
\begin{align}
\delta y_{0,m^*+1}^* &= \mu y_{1,m^*+1}^* \nonumber\\
(\mu + \delta) y_{i,m+1}^* &= \mu y_{i+1,m^*+1}^* + \nu y_{i-1,m^*}^*,\nonumber\\
&\hspace*{1in} i = 1, \dots, m^* \nonumber\\
\mu  y_{m^*+1,m^*+1}^* &= \nu y_{m^*,m^*}^*,\label{eq:25072}
\end{align}

or equivalently,
\[
(\mu + \delta) y_{m^*+1,m^*+1}^* = \nu y_{m^*,m^*}^* + \delta y_{m^*+1,m^*+1}^*,
\]
along with
\[
\sum_{i = 0}^{m^*} y_{i,m^*}^* + \sum_{i = 0}^{m^* + 1} y_{i,m^*+1}^* = 1.
\]

Note that Equations (\ref{eq:25071}) and (\ref{eq:25072}) determine $y_{m,m}^*$ and $y_{m+1,m+1}^*$: 
\begin{align*}
	y_{m,m}^*&=
	 \frac{\d}{(1+\nu)(1+\d)},\\
	y_{m+1,m+1}^*&=
	\frac{\nu\d}{(1+\nu)(1+\d)}.
\end{align*}

It follows from the above equations (flux up equals flux down) that
\begin{align*}
&\delta \sum_{i = 0}^{m^* - 1} (m^* - i) [y_{i,m^*}^* + y_{i,m^*+1}^*] +
\nu w_{m^*}^*\\ &=
\mu\left[ \sum_{i = 1}^{m^*} y_{i,m^*}^* + \sum_{i = 1}^{m^* + 1} y_{i,m^*+1}^*\right],
\end{align*}
which implies that
\[
\nu = \frac{\lambda - \Delta}{w_m^*},
\]
with
\[
\Delta = \delta \sum_{i = 0}^{m^* - 1} (m^* - i) [y_{i,m^*}^* + y_{i,m^*+1}^*],
\]
is equivalent with
\[
y_{0,m^*}^* + y_{0,m^*+1}^* = 1 - \frac{\lambda}{\mu},
\]
reflecting that each server is idle a fraction of the time
$1 - \lambda / \mu$.

Recall that $a=\frac{1}{1+\delta}$ and $b=\frac{1}{1+\delta+\nu}$.
We can use the above equations to express
$y_{m^*-j,m^*}$ in $y_{m^*-j+1,m^*+1}^*$ for all $j = 1, \dots, m^*$,
and recursively obtain
\[
\begin{aligned}
y_{m^*-j,m^*}^*
&=
b^j y_{m^*,m^*}^*, 
\hspace*{.1in} j = 0, \dots, m^* - 1 \\
y_{i,m^*}^*
&=
b^{m^* - i} y_{m^*,m^*},
\hspace*{.1in} i = 1, \dots, m^* \\
y_{0,m^*}^*
&=
y_{m^*,m^*}^*\\
& = \frac{b^{m^* - 1}}{\nu + \delta} y_{m^*,m^*}.
\end{aligned}
\]

Subsequently, we express $y_{m^*-j,m^*+1}^*$ in terms of
$y_{m^*-j+1,m^*+1}^*$ and $y_{m^*-j-1,m^*}^*$, and recursively
derive
\begin{align*}
y_{m^*+1,m^*+1}
=&
\nu y_{m^*,m^*} \\
y_{m^*-j,m^* + 1}^*
=&
a^{j+1} y_{m^*+1,m^*+1}^* + \nu a b
\sum_{k = 0}^j a^{j - k} b^k y_{m^*,m^*}^* \\
=&
a^{j+1} y_{m^*+1,m^*+1}^* + [a^{j+1} - b^{j+1}] y_{m^*,m^*}^* 
\end{align*}
for $j = - 1, \dots, m^* - 2$, 
\begin{align*}
y_{m^*+1-j,m^* + 1}^*
=&
a^j y_{m^*+1,m^*+1}^* + \nu a b
\sum_{k = 0}^{j - 1} a^{j - k} b^k y_{m^*,m^*}^* \\
=&
a^j y_{m^*+1,m^*+1}^* + [a^j - b^j] y_{m^*,m^*}^* 
\end{align*}
for $j = 0, \dots, m^* - 1$, 
\begin{align*}
y_{i,m^* + 1}^*
=&
a^{m^* + 1 - i} y_{m^*+1,m^*+1}^* + \nu a b
\sum_{k = 0}^{m^* - i} a^{m^* - i - k} b^k y_{m^*,m^*}^* \\
=&
a^{m^* + 1 - i} y_{m^*+1,m^*+1}^* + 
[a^{m^* - i + 1} - b^{m^* - i + 1}] y_{m^*,m^*}^* 
\end{align*}
for $i = 2, \dots, m^* + 1$, 
\begin{align*}
y_{1,m^*+1}^*
=&
a^{m^*} y_{m^*+1,m^*+1}^* + \nu a b
\left[\sum_{k = 0}^{m^* - 2} a^{m^* - 1 - k} b^k +
\frac{1}{\nu + \delta} b^{m^* - 2}\right] y_{m^*,m^*}^* \\
=&
a^{m^*} y_{m^*+1,m^*+1}^* +
a \left[a^{m^* - 1}  - \frac{\delta}{\nu + \delta} b^{m^* - 1}\right] 
y_{m^*,m^*}^*
\end{align*}
and $y_{0,m^*+1}^* = \frac{1}{\delta} y_{1,m^*+1}^*$.

It only remains to be shown that the equation \eqref{nueq} has
a unique solution $\nu \geq 0$, which then further implies that
\[
\nu = \frac{\lambda - \Delta}{w_{m^*}^*},
\]
as noted earlier.

In order to establish that a solution $\nu \geq 0$ exists, note that
$y_{0,m^*+1}^* \downarrow 0$, and
\[
y_{0,m^*} \to \left(\frac{1}{1 + \delta}\right)^{m^*} \leq 1 - \lambda,
\]
as $\nu \downarrow 0$, while $y_{0,m^*} \downarrow 0$ and
\[
y_{0,m^*+1} \to \left(\frac{1}{1 + \delta}\right)^{m^*+1} > 1 - \lambda
\]
as $\nu \to \infty$.

It may further be shown that $y_{0,m^*} + y_{0,m^*+1}$ is in fact
(strictly) decreasing in~$\nu$, ensuring that the value of~$\nu$ is
also unique.

\section{Simulation results for $\sujsqe$}
\label{app:simressync}

The next four figures provide the fluid-limit trajectories
and associated simulation paths for a system with $N = 1000$ servers
and $\lambda = 0.7$ for $\sujsqe$ as referred to in Section~\ref{simusync}.

\begin{figure}[h]
    \centering
    \begin{minipage}{0.460\textwidth}
        \centering
        \includegraphics[width=\textwidth]{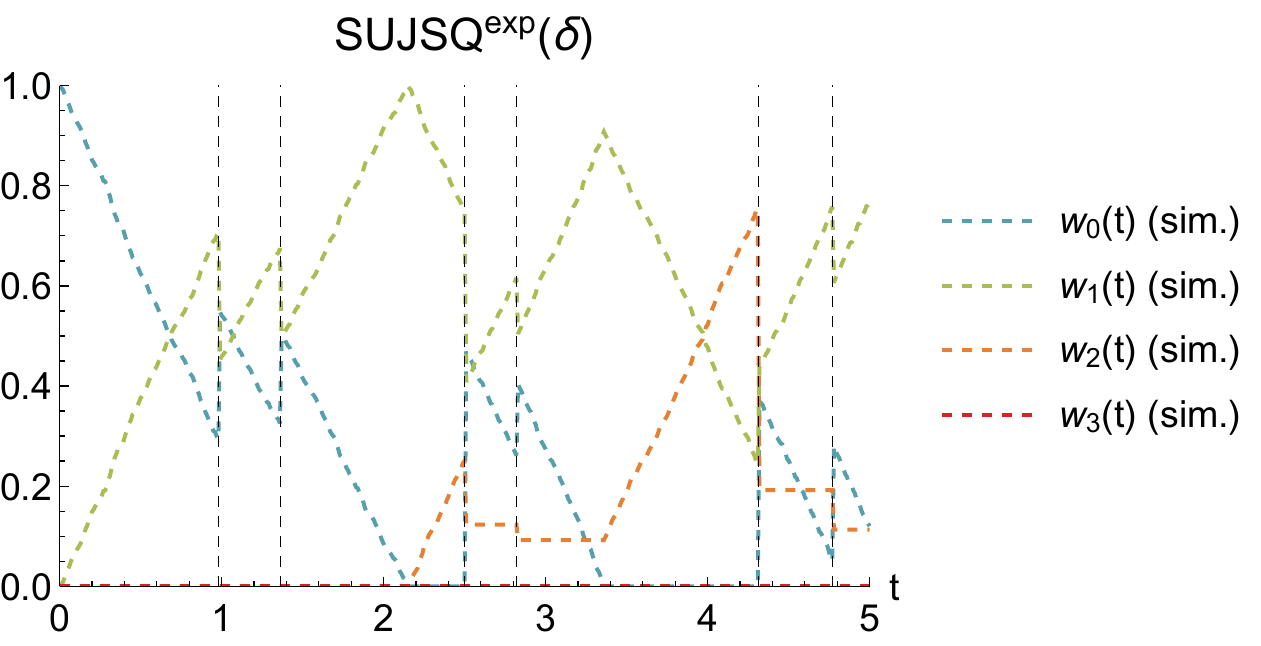} 
        \caption{Simulation results for SUJSQ$^{exp}(0.85)$.}
        	\label{fig:ss21}
    \end{minipage}\hfill
    \begin{minipage}{0.460\textwidth}
        \centering
        \includegraphics[width=\textwidth]{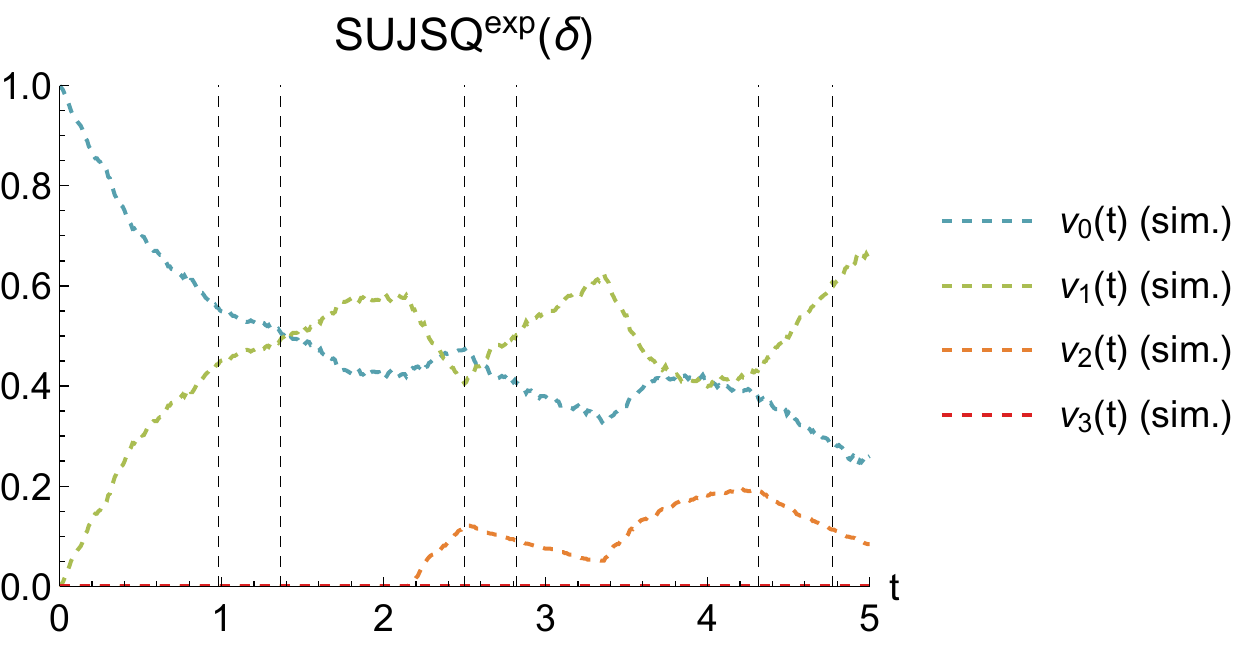} 
        \caption{Simulation results for SUJSQ$^{exp}(0.85)$.}
        	\label{fig:ss22}
    \end{minipage}
\end{figure}

\begin{figure}[h]
    \centering
    \begin{minipage}{0.460\textwidth}
        \centering
        \includegraphics[width=\textwidth]{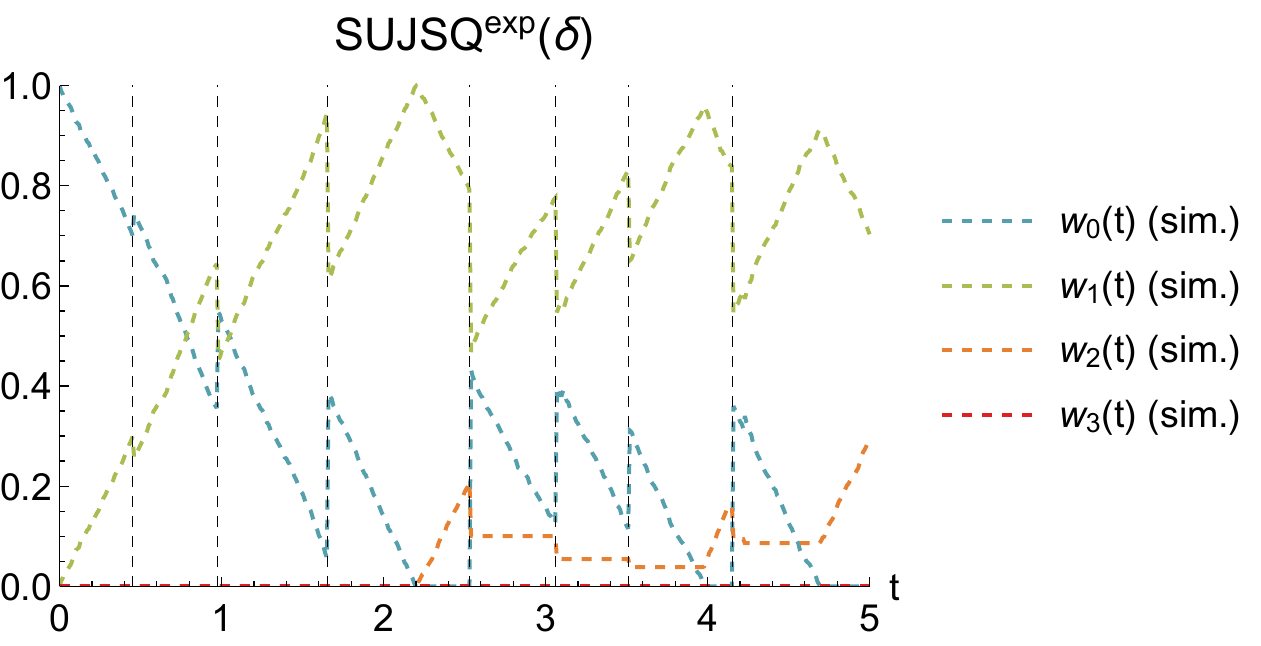} 
	\caption{Simulation results for SUJSQ$^{exp}(2.5)$.}
	\label{fig:ss23}
    \end{minipage}\hfill
    \begin{minipage}{0.460\textwidth}
        \centering
        \includegraphics[width=\textwidth]{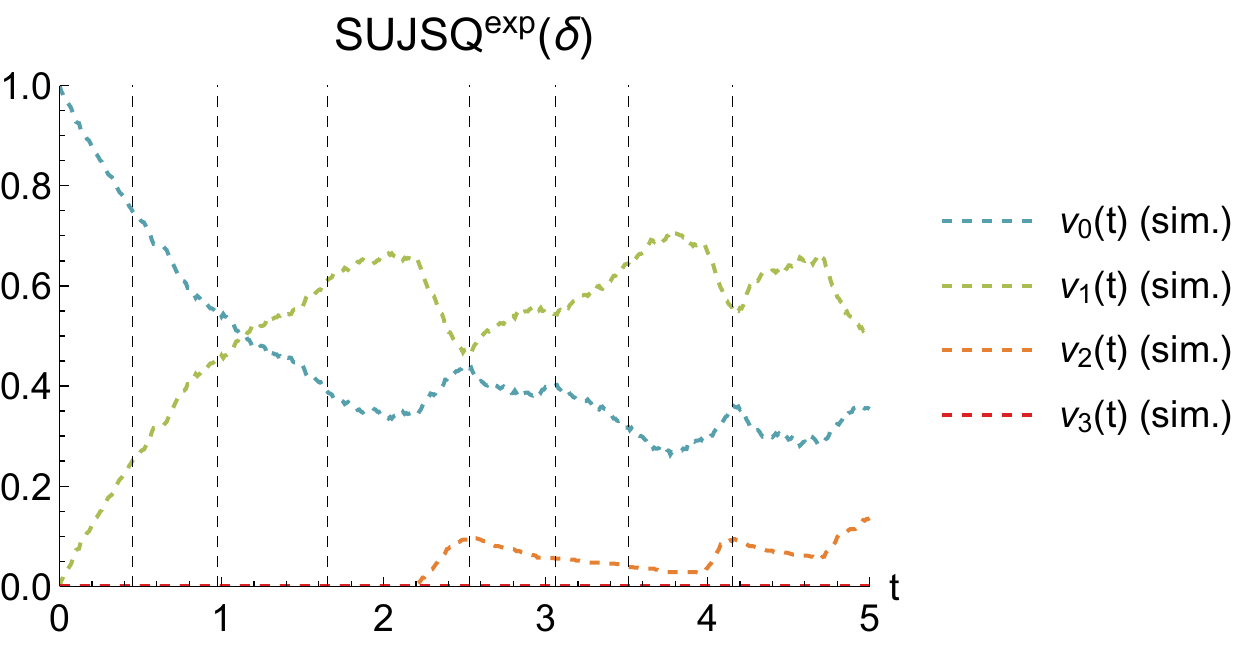} 
        \caption{Simulation results for SUJSQ$^{exp}(2.5)$.}
        	\label{fig:ss24}
    \end{minipage}
\end{figure}

\section{Simulation results for $\aujsqd$}
\label{app:simresasyn}

The next four figures provide the simulation plots for a system
with $N = 400$ servers and $\lambda = 0.7$, averaged over 100 runs for $\aujsqd$ as referred
to in Section~\ref{simuasyn}.

\begin{figure}[h]
    \centering
    \begin{minipage}{0.460\textwidth}
        \centering
        \includegraphics[width=\textwidth]{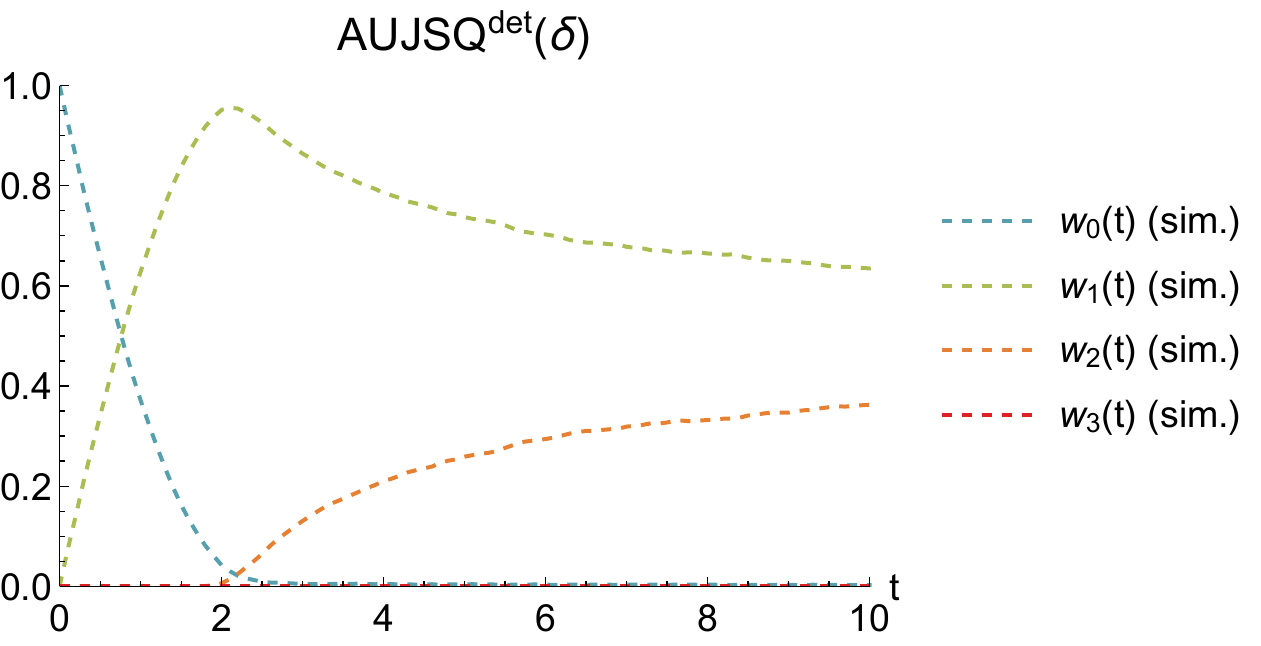} 
        \caption{Simulation results for AUJSQ$^{det}(0.85)$.}\label{fig:ss41}
    \end{minipage}\hfill
    \begin{minipage}{0.460\textwidth}
        \centering
        \includegraphics[width=\textwidth]{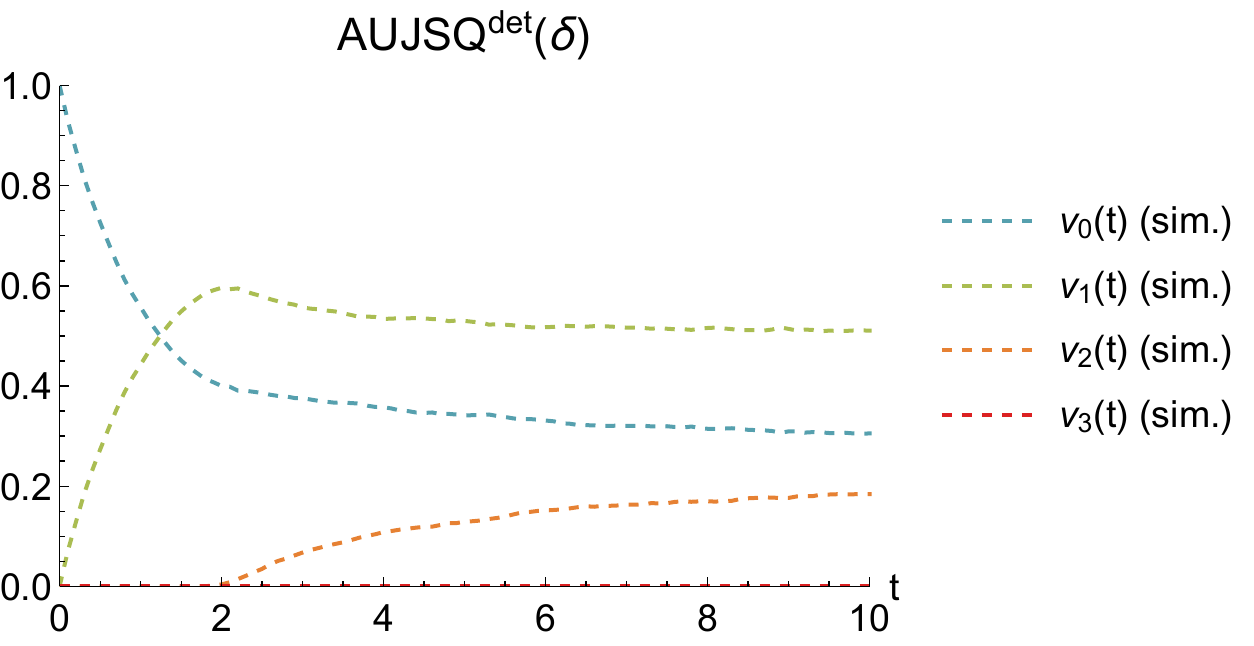} 
        \caption{Simulation results for AUJSQ$^{det}(0.85)$.}\label{fig:ss42}
    \end{minipage}
\end{figure}

\begin{figure}[h]
    \centering
    \begin{minipage}{0.460\textwidth}
        \centering
        \includegraphics[width=\textwidth]{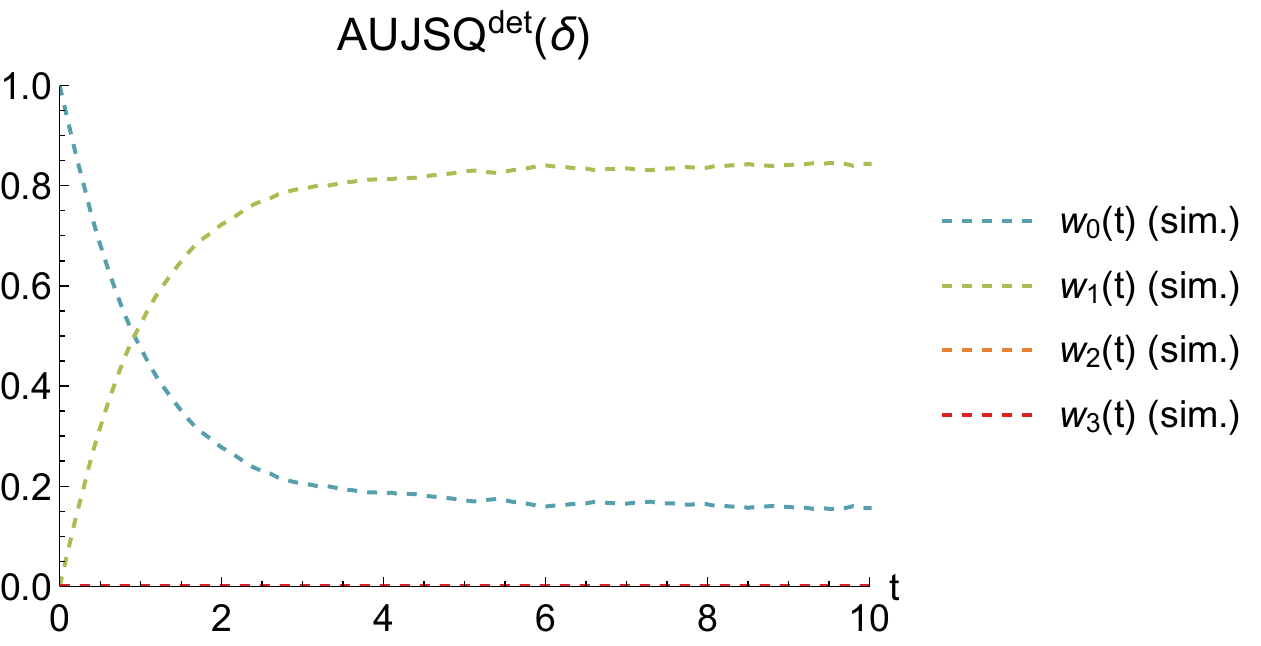} 
        \caption{Simulation results for AUJSQ$^{det}(2.5)$.}\label{fig:ss43}
    \end{minipage}\hfill
    \begin{minipage}{0.460\textwidth}
        \centering
        \includegraphics[width=\textwidth]{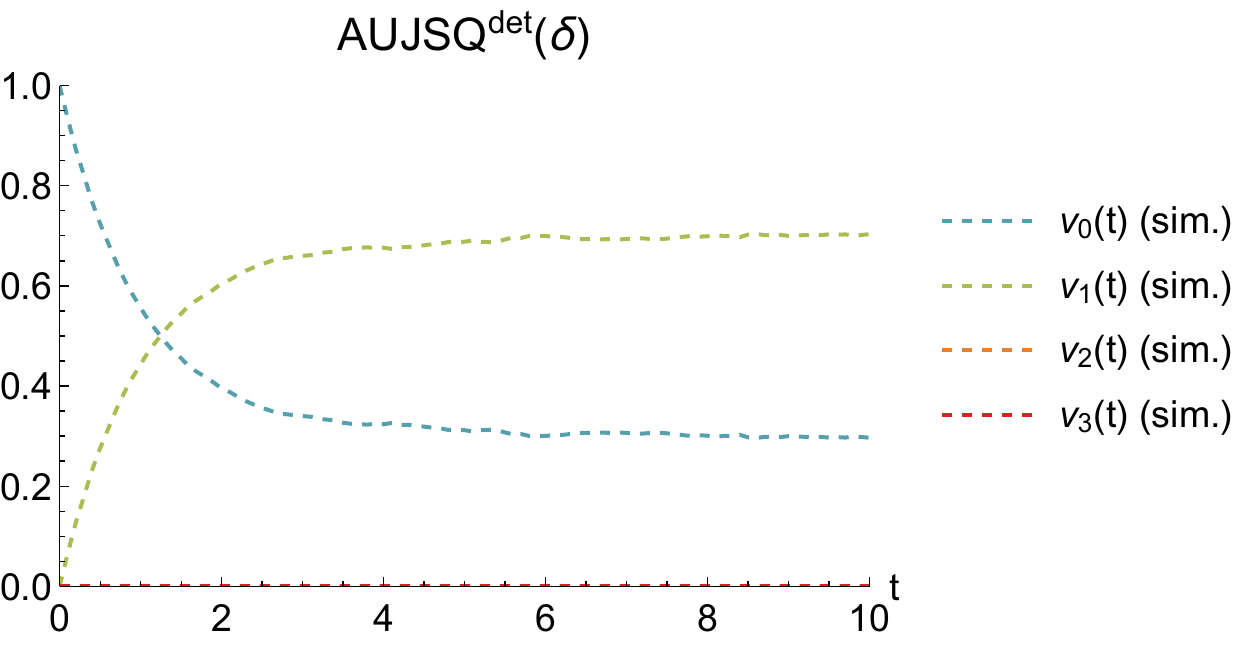} 
        \caption{Simulation results for AUJSQ$^{det}(2.5)$.}\label{fig:ss44}
    \end{minipage}
\end{figure}

\bibliographystyle{ACM-Reference-Format}
\bibliography{biblioINI}


\begin{thebibliography}{31}


\ifx \showCODEN    \undefined \def \showCODEN     #1{\unskip}     \fi
\ifx \showDOI      \undefined \def \showDOI       #1{#1}\fi
\ifx \showISBNx    \undefined \def \showISBNx     #1{\unskip}     \fi
\ifx \showISBNxiii \undefined \def \showISBNxiii  #1{\unskip}     \fi
\ifx \showISSN     \undefined \def \showISSN      #1{\unskip}     \fi
\ifx \showLCCN     \undefined \def \showLCCN      #1{\unskip}     \fi
\ifx \shownote     \undefined \def \shownote      #1{#1}          \fi
\ifx \showarticletitle \undefined \def \showarticletitle #1{#1}   \fi
\ifx \showURL      \undefined \def \showURL       {\relax}        \fi
\providecommand\bibfield[2]{#2}
\providecommand\bibinfo[2]{#2}
\providecommand\natexlab[1]{#1}
\providecommand\showeprint[2][]{arXiv:#2}

\bibitem[\protect\citeauthoryear{Alon, Gurel-Gurevich, and Lubetzky}{Alon
  et~al\mbox{.}}{2010}]%
        {AGL10}
\bibfield{author}{\bibinfo{person}{N Alon}, \bibinfo{person}{O Gurel-Gurevich},
  {and} \bibinfo{person}{E Lubetzky}.} \bibinfo{year}{2010}\natexlab{}.
\newblock \showarticletitle{Choice-memory tradeoff in allocations}.
\newblock \bibinfo{journal}{\emph{Ann. Appl. Probab.}} \bibinfo{volume}{20},
  \bibinfo{number}{4} (\bibinfo{date}{08} \bibinfo{year}{2010}),
  \bibinfo{pages}{1470--1511}.
\newblock
\urldef\tempurl%
\url{https://doi.org/10.1214/09-AAP656}
\showDOI{\tempurl}


\bibitem[\protect\citeauthoryear{Anselmi and Dufour}{Anselmi and
  Dufour}{2018}]%
        {AD18}
\bibfield{author}{\bibinfo{person}{J Anselmi} {and} \bibinfo{person}{F
  Dufour}.} \bibinfo{year}{2018}\natexlab{}.
\newblock \showarticletitle{Power-of-$d$-Choices with Memory: Fluid Limit and
  Optimality}.
\newblock \bibinfo{journal}{\emph{arXiv preprint arXiv:1802.06566}}
  (\bibinfo{year}{2018}).
\newblock


\bibitem[\protect\citeauthoryear{Badonnel and Burgess}{Badonnel and
  Burgess}{2008}]%
        {BB08}
\bibfield{author}{\bibinfo{person}{R Badonnel} {and} \bibinfo{person}{M
  Burgess}.} \bibinfo{year}{2008}\natexlab{}.
\newblock \showarticletitle{Dynamic pull-based load balancing for autonomic
  servers}. In \bibinfo{booktitle}{\emph{Network Operations and Management
  Symposium, 2008. NOMS 2008. IEEE}}. IEEE, \bibinfo{pages}{751--754}.
\newblock


\bibitem[\protect\citeauthoryear{Bramson, Lu, and Prabhakar}{Bramson
  et~al\mbox{.}}{2010}]%
        {BLP10}
\bibfield{author}{\bibinfo{person}{M Bramson}, \bibinfo{person}{Y Lu}, {and}
  \bibinfo{person}{B Prabhakar}.} \bibinfo{year}{2010}\natexlab{}.
\newblock \showarticletitle{Randomized load balancing with general service time
  distributions}. In \bibinfo{booktitle}{\emph{ACM SIGMETRICS Performance
  Evaluation Review}}, Vol.~\bibinfo{volume}{38(1)}. ACM,
  \bibinfo{pages}{275--286}.
\newblock


\bibitem[\protect\citeauthoryear{Bramson, Lu, and Prabhakar}{Bramson
  et~al\mbox{.}}{2012}]%
        {BLP12}
\bibfield{author}{\bibinfo{person}{M Bramson}, \bibinfo{person}{Y Lu}, {and}
  \bibinfo{person}{B Prabhakar}.} \bibinfo{year}{2012}\natexlab{}.
\newblock \showarticletitle{Asymptotic independence of queues under randomized
  load balancing}.
\newblock \bibinfo{journal}{\emph{Queueing Systems}} \bibinfo{volume}{71},
  \bibinfo{number}{3} (\bibinfo{year}{2012}), \bibinfo{pages}{247--292}.
\newblock


\bibitem[\protect\citeauthoryear{Ephremides, Varaiya, and Walrand}{Ephremides
  et~al\mbox{.}}{1980}]%
        {EVW80}
\bibfield{author}{\bibinfo{person}{A Ephremides}, \bibinfo{person}{P Varaiya},
  {and} \bibinfo{person}{J Walrand}.} \bibinfo{year}{1980}\natexlab{}.
\newblock \showarticletitle{A simple dynamic routing problem}.
\newblock \bibinfo{journal}{\emph{IEEE transactions on Automatic Control}}
  \bibinfo{volume}{25}, \bibinfo{number}{4} (\bibinfo{year}{1980}),
  \bibinfo{pages}{690--693}.
\newblock


\bibitem[\protect\citeauthoryear{Gamarnik, Tsitsiklis, and Zubeldia}{Gamarnik
  et~al\mbox{.}}{2016}]%
        {GTZ16}
\bibfield{author}{\bibinfo{person}{D Gamarnik}, \bibinfo{person}{J~N
  Tsitsiklis}, {and} \bibinfo{person}{M Zubeldia}.}
  \bibinfo{year}{2016}\natexlab{}.
\newblock \showarticletitle{Delay, memory, and messaging tradeoffs in
  distributed service systems}.
\newblock \bibinfo{journal}{\emph{ACM SIGMETRICS Performance Evaluation
  Review}} \bibinfo{volume}{44}, \bibinfo{number}{1} (\bibinfo{year}{2016}),
  \bibinfo{pages}{1--12}.
\newblock


\bibitem[\protect\citeauthoryear{Gandhi, Liu, Hu, Lu, Padhye, Yuan, and
  Zhang}{Gandhi et~al\mbox{.}}{2014}]%
        {duet14}
\bibfield{author}{\bibinfo{person}{R Gandhi}, \bibinfo{person}{H~H Liu},
  \bibinfo{person}{Y~C Hu}, \bibinfo{person}{G Lu}, \bibinfo{person}{J Padhye},
  \bibinfo{person}{L Yuan}, {and} \bibinfo{person}{M Zhang}.}
  \bibinfo{year}{2014}\natexlab{}.
\newblock \showarticletitle{Duet: Cloud scale load balancing with hardware and
  software}.
\newblock \bibinfo{journal}{\emph{ACM SIGCOMM Computer Communication Review}}
  \bibinfo{volume}{44}, \bibinfo{number}{4} (\bibinfo{year}{2014}),
  \bibinfo{pages}{27--38}.
\newblock


\bibitem[\protect\citeauthoryear{Gast}{Gast}{2017}]%
        {Gast17}
\bibfield{author}{\bibinfo{person}{N Gast}.} \bibinfo{year}{2017}\natexlab{}.
\newblock \showarticletitle{Expected values estimated via mean-field
  approximation are 1/N-accurate}.
\newblock \bibinfo{journal}{\emph{Proceedings of the ACM on Measurement and
  Analysis of Computing Systems}} \bibinfo{volume}{1}, \bibinfo{number}{1}
  (\bibinfo{year}{2017}), \bibinfo{pages}{17}.
\newblock


\bibitem[\protect\citeauthoryear{Gupta and Walton}{Gupta and Walton}{2019}]%
        {GW17}
\bibfield{author}{\bibinfo{person}{V Gupta} {and} \bibinfo{person}{N Walton}.}
  \bibinfo{year}{2019}\natexlab{}.
\newblock \showarticletitle{Load Balancing in the Nondegenerate Slowdown
  Regime}.
\newblock \bibinfo{journal}{\emph{Operations Research}} (\bibinfo{year}{2019}).
\newblock


\bibitem[\protect\citeauthoryear{Hunt and Kurtz}{Hunt and Kurtz}{1994}]%
        {HK94}
\bibfield{author}{\bibinfo{person}{P~J Hunt} {and} \bibinfo{person}{T~G
  Kurtz}.} \bibinfo{year}{1994}\natexlab{}.
\newblock \showarticletitle{Large loss networks}.
\newblock \bibinfo{journal}{\emph{Stochastic Processes and their Applications}}
  \bibinfo{volume}{53}, \bibinfo{number}{2} (\bibinfo{year}{1994}),
  \bibinfo{pages}{363--378}.
\newblock


\bibitem[\protect\citeauthoryear{Lu, Xie, Kliot, Geller, Larus, and
  Greenberg}{Lu et~al\mbox{.}}{2011}]%
        {LXKGLG11}
\bibfield{author}{\bibinfo{person}{Y Lu}, \bibinfo{person}{Q Xie},
  \bibinfo{person}{G Kliot}, \bibinfo{person}{A Geller}, \bibinfo{person}{J~R
  Larus}, {and} \bibinfo{person}{A Greenberg}.}
  \bibinfo{year}{2011}\natexlab{}.
\newblock \showarticletitle{Join-Idle-Queue: A novel load balancing algorithm
  for dynamically scalable web services}.
\newblock \bibinfo{journal}{\emph{Performance Evaluation}}
  \bibinfo{volume}{68}, \bibinfo{number}{11} (\bibinfo{year}{2011}),
  \bibinfo{pages}{1056--1071}.
\newblock


\bibitem[\protect\citeauthoryear{Luczak and Norris}{Luczak and Norris}{2013}]%
        {LN13}
\bibfield{author}{\bibinfo{person}{M~J Luczak} {and} \bibinfo{person}{J~R
  Norris}.} \bibinfo{year}{2013}\natexlab{}.
\newblock \showarticletitle{Averaging over fast variables in the fluid limit
  for Markov chains: application to the supermarket model with memory}.
\newblock \bibinfo{journal}{\emph{The Annals of Applied Probability}}
  \bibinfo{volume}{23}, \bibinfo{number}{3} (\bibinfo{year}{2013}),
  \bibinfo{pages}{957--986}.
\newblock


\bibitem[\protect\citeauthoryear{Maguluri, Srikant, and Ying}{Maguluri
  et~al\mbox{.}}{2012}]%
        {MSY12}
\bibfield{author}{\bibinfo{person}{S~T Maguluri}, \bibinfo{person}{R Srikant},
  {and} \bibinfo{person}{L Ying}.} \bibinfo{year}{2012}\natexlab{}.
\newblock \showarticletitle{Stochastic models of load balancing and scheduling
  in cloud computing clusters}. In \bibinfo{booktitle}{\emph{INFOCOM, 2012
  Proceedings IEEE}}. IEEE, \bibinfo{pages}{702--710}.
\newblock


\bibitem[\protect\citeauthoryear{Mitzenmacher}{Mitzenmacher}{2000}]%
        {Mitzenmacher00}
\bibfield{author}{\bibinfo{person}{M Mitzenmacher}.}
  \bibinfo{year}{2000}\natexlab{}.
\newblock \showarticletitle{How useful is old information?}
\newblock \bibinfo{journal}{\emph{IEEE Transactions on Parallel and Distributed
  Systems}} \bibinfo{volume}{11}, \bibinfo{number}{1} (\bibinfo{year}{2000}),
  \bibinfo{pages}{6--20}.
\newblock


\bibitem[\protect\citeauthoryear{Mitzenmacher}{Mitzenmacher}{2001}]%
        {Mitzenmacher01}
\bibfield{author}{\bibinfo{person}{M Mitzenmacher}.}
  \bibinfo{year}{2001}\natexlab{}.
\newblock \showarticletitle{The power of two choices in randomized load
  balancing}.
\newblock \bibinfo{journal}{\emph{IEEE Transactions on Parallel and Distributed
  Systems}} \bibinfo{volume}{12}, \bibinfo{number}{10} (\bibinfo{year}{2001}),
  \bibinfo{pages}{1094--1104}.
\newblock


\bibitem[\protect\citeauthoryear{Mitzenmacher, Prabhakar, and
  Shah}{Mitzenmacher et~al\mbox{.}}{2002}]%
        {Mitzenmacher02}
\bibfield{author}{\bibinfo{person}{M Mitzenmacher}, \bibinfo{person}{B
  Prabhakar}, {and} \bibinfo{person}{D Shah}.} \bibinfo{year}{2002}\natexlab{}.
\newblock \showarticletitle{Load balancing with memory}. In
  \bibinfo{booktitle}{\emph{The 43rd Annual IEEE Symposium on Foundations of
  Computer Science, 2002. Proceedings.}} \bibinfo{pages}{799--808}.
\newblock
\showISSN{0272-5428}
\urldef\tempurl%
\url{https://doi.org/10.1109/SFCS.2002.1182005}
\showDOI{\tempurl}


\bibitem[\protect\citeauthoryear{Mukherjee, Borst, Van~Leeuwaarden, and
  Whiting}{Mukherjee et~al\mbox{.}}{2016}]%
        {MBLW16}
\bibfield{author}{\bibinfo{person}{D Mukherjee}, \bibinfo{person}{S~C Borst},
  \bibinfo{person}{J~SH Van~Leeuwaarden}, {and} \bibinfo{person}{P~A Whiting}.}
  \bibinfo{year}{2016}\natexlab{}.
\newblock \showarticletitle{Universality of load balancing schemes on the
  diffusion scale}.
\newblock \bibinfo{journal}{\emph{Journal of Applied Probability}}
  \bibinfo{volume}{53}, \bibinfo{number}{4} (\bibinfo{year}{2016}),
  \bibinfo{pages}{1111--1124}.
\newblock


\bibitem[\protect\citeauthoryear{Mukhopadhyay, Karthik, and
  Mazumdar}{Mukhopadhyay et~al\mbox{.}}{2016}]%
        {MKM15}
\bibfield{author}{\bibinfo{person}{A Mukhopadhyay}, \bibinfo{person}{A
  Karthik}, {and} \bibinfo{person}{R~R Mazumdar}.}
  \bibinfo{year}{2016}\natexlab{}.
\newblock \showarticletitle{Randomized assignment of jobs to servers in
  heterogeneous clusters of shared servers for low delay}.
\newblock \bibinfo{journal}{\emph{Stochastic Systems}} \bibinfo{volume}{6},
  \bibinfo{number}{1} (\bibinfo{year}{2016}), \bibinfo{pages}{90--131}.
\newblock


\bibitem[\protect\citeauthoryear{Mukhopadhyay, Karthik, Mazumdar, and
  Guillemin}{Mukhopadhyay et~al\mbox{.}}{2015}]%
        {MKMG15}
\bibfield{author}{\bibinfo{person}{A Mukhopadhyay}, \bibinfo{person}{A
  Karthik}, \bibinfo{person}{R~R Mazumdar}, {and} \bibinfo{person}{F
  Guillemin}.} \bibinfo{year}{2015}\natexlab{}.
\newblock \showarticletitle{Mean field and propagation of chaos in multi-class
  heterogeneous loss models}.
\newblock \bibinfo{journal}{\emph{Performance Evaluation}}
  \bibinfo{volume}{91} (\bibinfo{year}{2015}), \bibinfo{pages}{117--131}.
\newblock


\bibitem[\protect\citeauthoryear{Mukhopadhyay and Mazumdar}{Mukhopadhyay and
  Mazumdar}{2014}]%
        {MM14a}
\bibfield{author}{\bibinfo{person}{A Mukhopadhyay} {and} \bibinfo{person}{R~R
  Mazumdar}.} \bibinfo{year}{2014}\natexlab{}.
\newblock \showarticletitle{Randomized routing schemes for large processor
  sharing systems with multiple service rates}. In
  \bibinfo{booktitle}{\emph{ACM SIGMETRICS Performance Evaluation Review}},
  Vol.~\bibinfo{volume}{42(1)}. ACM, \bibinfo{pages}{555--556}.
\newblock


\bibitem[\protect\citeauthoryear{Pang, Talreja, and Whitt}{Pang
  et~al\mbox{.}}{2007}]%
        {PTW07}
\bibfield{author}{\bibinfo{person}{G Pang}, \bibinfo{person}{R Talreja}, {and}
  \bibinfo{person}{W Whitt}.} \bibinfo{year}{2007}\natexlab{}.
\newblock \showarticletitle{Martingale proofs of many-server heavy-traffic
  limits for Markovian queues}.
\newblock \bibinfo{journal}{\emph{Probability Surveys}}  \bibinfo{volume}{4}
  (\bibinfo{year}{2007}), \bibinfo{pages}{193--267}.
\newblock


\bibitem[\protect\citeauthoryear{Patel, Bansal, Yuan, Murthy, Greenberg, Maltz,
  Kern, Kumar, Zikos, Wu, Changhoon, and Karri}{Patel et~al\mbox{.}}{2013}]%
        {ananta13}
\bibfield{author}{\bibinfo{person}{P Patel}, \bibinfo{person}{D Bansal},
  \bibinfo{person}{L Yuan}, \bibinfo{person}{A Murthy}, \bibinfo{person}{A
  Greenberg}, \bibinfo{person}{D~A Maltz}, \bibinfo{person}{R Kern},
  \bibinfo{person}{H Kumar}, \bibinfo{person}{M Zikos}, \bibinfo{person}{H Wu},
  \bibinfo{person}{K Changhoon}, {and} \bibinfo{person}{N Karri}.}
  \bibinfo{year}{2013}\natexlab{}.
\newblock \showarticletitle{Ananta: Cloud scale load balancing}.
\newblock \bibinfo{journal}{\emph{ACM SIGCOMM Computer Communication Review}}
  \bibinfo{volume}{43}, \bibinfo{number}{4} (\bibinfo{year}{2013}),
  \bibinfo{pages}{207--218}.
\newblock


\bibitem[\protect\citeauthoryear{Stolyar}{Stolyar}{2015}]%
        {Stolyar15}
\bibfield{author}{\bibinfo{person}{A~L Stolyar}.}
  \bibinfo{year}{2015}\natexlab{}.
\newblock \showarticletitle{Pull-based load distribution in large-scale
  heterogeneous service systems}.
\newblock \bibinfo{journal}{\emph{Queueing Systems}} \bibinfo{volume}{80},
  \bibinfo{number}{4} (\bibinfo{year}{2015}), \bibinfo{pages}{341--361}.
\newblock


\bibitem[\protect\citeauthoryear{Tsitsiklis and Xu}{Tsitsiklis and Xu}{2012}]%
        {TX12}
\bibfield{author}{\bibinfo{person}{J~N Tsitsiklis} {and} \bibinfo{person}{K
  Xu}.} \bibinfo{year}{2012}\natexlab{}.
\newblock \showarticletitle{On the power of (even a little) resource pooling}.
\newblock \bibinfo{journal}{\emph{Stochastic Systems}} \bibinfo{volume}{2},
  \bibinfo{number}{1} (\bibinfo{year}{2012}), \bibinfo{pages}{1--66}.
\newblock


\bibitem[\protect\citeauthoryear{Tsitsiklis and Xu}{Tsitsiklis and Xu}{2013}]%
        {TX13}
\bibfield{author}{\bibinfo{person}{J~N Tsitsiklis} {and} \bibinfo{person}{K
  Xu}.} \bibinfo{year}{2013}\natexlab{}.
\newblock \showarticletitle{Queueing system topologies with limited
  flexibility}. In \bibinfo{booktitle}{\emph{ACM SIGMETRICS Performance
  Evaluation Review}}, Vol.~\bibinfo{volume}{41(1)}. ACM,
  \bibinfo{pages}{167--178}.
\newblock


\bibitem[\protect\citeauthoryear{Vvedenskaya, Dobrushin, and
  Karpelevich}{Vvedenskaya et~al\mbox{.}}{1996}]%
        {VDK96}
\bibfield{author}{\bibinfo{person}{N~D Vvedenskaya}, \bibinfo{person}{R~L
  Dobrushin}, {and} \bibinfo{person}{F~I Karpelevich}.}
  \bibinfo{year}{1996}\natexlab{}.
\newblock \showarticletitle{Queueing system with selection of the shortest of
  two queues: An asymptotic approach}.
\newblock \bibinfo{journal}{\emph{Problemy Peredachi Informatsii}}
  \bibinfo{volume}{32}, \bibinfo{number}{1} (\bibinfo{year}{1996}),
  \bibinfo{pages}{20--34}.
\newblock


\bibitem[\protect\citeauthoryear{Winston}{Winston}{1977}]%
        {Winston77}
\bibfield{author}{\bibinfo{person}{W Winston}.}
  \bibinfo{year}{1977}\natexlab{}.
\newblock \showarticletitle{Optimality of the shortest line discipline}.
\newblock \bibinfo{journal}{\emph{Journal of Applied Probability}}
  \bibinfo{volume}{14}, \bibinfo{number}{1} (\bibinfo{year}{1977}),
  \bibinfo{pages}{181--189}.
\newblock


\bibitem[\protect\citeauthoryear{Xie, Dong, Lu, and Srikant}{Xie
  et~al\mbox{.}}{2015}]%
        {XDLS15}
\bibfield{author}{\bibinfo{person}{Q Xie}, \bibinfo{person}{X Dong},
  \bibinfo{person}{Y Lu}, {and} \bibinfo{person}{R Srikant}.}
  \bibinfo{year}{2015}\natexlab{}.
\newblock \showarticletitle{Power of d choices for large-scale bin packing: A
  loss model}.
\newblock \bibinfo{journal}{\emph{ACM SIGMETRICS Performance Evaluation
  Review}} \bibinfo{volume}{43}, \bibinfo{number}{1} (\bibinfo{year}{2015}),
  \bibinfo{pages}{321--334}.
\newblock


\bibitem[\protect\citeauthoryear{Ying}{Ying}{2016}]%
        {Ying16}
\bibfield{author}{\bibinfo{person}{L Ying}.} \bibinfo{year}{2016}\natexlab{}.
\newblock \showarticletitle{On the approximation error of mean-field models}.
  In \bibinfo{booktitle}{\emph{ACM SIGMETRICS Performance Evaluation Review}},
  Vol.~\bibinfo{volume}{44(1)}. ACM, \bibinfo{pages}{285--297}.
\newblock


\bibitem[\protect\citeauthoryear{Ying, Srikant, and Kang}{Ying
  et~al\mbox{.}}{2015}]%
        {YSK15}
\bibfield{author}{\bibinfo{person}{L Ying}, \bibinfo{person}{R Srikant}, {and}
  \bibinfo{person}{X Kang}.} \bibinfo{year}{2015}\natexlab{}.
\newblock \showarticletitle{The power of slightly more than one sample in
  randomized load balancing}. In \bibinfo{booktitle}{\emph{Computer
  Communications (INFOCOM), 2015 IEEE Conference on}}. IEEE,
  \bibinfo{pages}{1131--1139}.
\newblock


\end{thebibliography}

\end{document}